\documentclass[12pt]{amsart}

\usepackage{amssymb, amscd}
\usepackage{latexsym,epsfig}
\usepackage[all]{xy}
\usepackage{pst-all}
\numberwithin{equation}{section}
\def\today{\ifcase\month\or Jan\or Febr\or  Mar\or  Apr\or May\or Jun\or  Jul\or Aug\or  Sep\or  Oct\or Nov\or  Dec\or\fi \space\number\day, \number\year}

\newcommand{\Sy}{{\mathrm{Sym}}} 
\newcommand{\CC}{\mathbb C}
\newcommand{\EE}{\mathbb E}
\newcommand{\FF}{\mathbb F}

\newcommand{\NN}{\mathbb N}

\newcommand{\QQ}{\mathbb Q}

\newcommand{\ZZ}{\mathbb Z}

\newcommand{\GL}{\rm GL}
\newcommand{\Sym}{\rm Sym}
\numberwithin{equation}{section}

\newcommand{\bigxor}{\mathop{\mathchoice
  {\textstyle\bigoplus}{\textstyle\bigoplus}
  {\scriptstyle\bigoplus}{\scriptscriptstyle\bigoplus}}}

\newtheorem{theorem}{Theorem}[section]

\newtheorem{proposition}[theorem]{Proposition}
\newtheorem{corollary}[theorem]{Corollary}

\newtheorem{definition-lemma}[theorem]{Definition-Lemma}
\theoremstyle{definition}

\theoremstyle{remark}
\newtheorem{remark}[theorem]{Remark}

\newtheorem{question}[theorem]{Question}

\begin{document}

\title[]{Constructing vector-valued Siegel modular forms from scalar-valued Siegel modular forms}
\author{Fabien Cl\'ery}
\address{Department Mathematik,
Universit\"at Siegen,
Emmy-Noether-Campus,
Walter-Flex-Strasse 3,
57068 Siegen, Germany}
\email{cleryfabien@gmail.com}

\author{Gerard van der Geer}
\address{Korteweg-de Vries Instituut, Universiteit van
Amsterdam, Postbus 94248,
1090 GE  Amsterdam, The Netherlands.}
\email{geer@science.uva.nl}

\begin{abstract}
This paper gives a simple method for constructing vector-valued Siegel modular forms
from scalar-valued ones. The method is efficient in producing 
the siblings of Delta, the  smallest weight cusp forms that appear 
in low degrees. 
It also shows the strong relations between these modular
forms of different degrees.
We illustrate this by a number of examples.
\end{abstract}

\maketitle
\begin{section}{Introduction}\label{sec-intro}
In this paper we describe a simple method to construct 
vector-valued Siegel modular
forms from scalar-valued ones by developing modular forms in the
normal bundle of a locus on which they vanish.
If $f$ is a scalar-valued Siegel modular form of degree $g$ and weight $k$,
given by a holomorphic function on the Siegel upper half space 
$\mathfrak{H}_g$ and $1\leq j\leq g-1$ an integer, then the restriction
of $f$ to the diagonally embedded $\frak{H}_{j} \times \frak{H}_{g-j}$ in
$\mathfrak{H}_g$
gives a tensor product $f' \otimes f^{\prime \prime}$ of modular
forms of weight $k$ and degree $j$ and $g-j$. If $f$ vanishes on 
$\mathfrak{H}_{j} \times \mathfrak{H}_{g-j}$, 
then we can develop $f$ in the normal bundle of 
$\mathfrak{H}_{j} \times \mathfrak{H}_{g-j}$ inside $\mathfrak{H}_g$
and one finds as lowest non-zero term a sum of tensor products 
of vector-valued Siegel  modular forms of degree $j$ and $g-j$. 
By applying this to well-known scalar-valued Siegel modular forms one
can produce explicit vector-valued Siegel modular forms on the full group
${\rm Sp}(2g,{\ZZ})$. 
In this way one can produce for example many siblings of Delta,
modular forms that play for a given low degree $g$ a role analogous 
to the role that
$\Delta$, the cusp form of weight $12$ on ${\rm SL}(2,{\ZZ})$, 
plays for elliptic modular
forms. That is, the cusp forms that appear among the first few 
if one orders these according to their Deligne weight. 
For example, for degree $2$ the first cusp forms that appear 
are a cusp form $\chi_{10}$ of weight
$10$, a form $\chi_{12}$ of weight $12$ 
and a vector-valued form $\chi_{6,8}$ of weight $(j,k)=(6,8)$ of
Deligne weight $j+2k-3$ equal to $17$, $21$ and $19$.  
The forms $\chi_{10}$ and $\chi_{6,8}$ appear immediately
in the development of the Schottky form, 
a scalar-valued cusp form of weight $8$
in degree $4$, along $\mathfrak{H}_2 \times \mathfrak{H}_2$.  
Similarly, the first cusp form in degree $3$ is a 
cusp form of weight $(4,0,8)$ and is  also obtained by developing
the Schottky form, this time
along $\mathfrak{H}_1 \times \mathfrak{H}_3$.

Recently there has been quite some progress in our knowledge
of vector-valued Siegel modular forms of low degree; besides 
\cite{B-F-vdG} there is the impressive 
work of Chenevier and Renard and Ta\"{\i}bi (see \cite{C, C-R, Taibi})
who among other things give
conjectural values for the dimensions of spaces of cusp forms for 
degree $g \leq 6$.  As an illustration of our method  we 
construct cusp forms in a number of 
cases where the dimension of the space is predicted to be~$1$. 
For example, we construct the `first' cusp form in degree $5$,
of weight $(2,0,0,0,10)$.
Likewise, we construct a form of weight $(2,0,0,0,0,0,10)$ 
in degree~$7$, again of rather low Deligne weight ($44$) compared with
that ($56$) of the first scalar-valued cusp form (of weight $12$).

Just like the ubiquitous $\Delta$, the  siblings of $\Delta$ appear 
to play a role at many places, e.g.\ in the study of K3 surfaces, 
and they all appear to be intimately connected.
We study the form $\chi_{6,8}$ and some other siblings of Delta 
in more detail, giving alternative constructions
and calculating eigenvalues of Hecke operators. 
In an appendix we summarize some facts about Hecke operators of 
degree~$2$ and~$3$ that we need.

\end{section}

\begin{section}{Restricting scalar-valued modular forms}\label{restriction}
Let $\Gamma_g={\rm Sp}(2g,{\ZZ})$ be the symplectic group of degree $g$
acting in the usual way on the Siegel upper half space $\frak{H}_g$:
$$
\tau \mapsto (a\tau+b)(c\tau+d)^{-1} 
\qquad \text{for all $\tau\in \frak{H}_g$ and
$\left( \begin{smallmatrix} a & b \\ c & d \\ \end{smallmatrix} \right) 
\in \Gamma_g$ .} 
$$
We denote by $M_k(\Gamma_g)$ the space of scalar-valued Siegel modular forms
of weight $k$ on $\Gamma_g$, that is, holomorphic functions on $\frak{H}_g$ satisfying
$$
f((a\tau+b)(c\tau+d)^{-1})= \det (c\tau+d)^k f(\tau)
$$
for all $(a,b;c,d) \in \Gamma_g$ and
$\tau\in \frak{H}_g$ (and satisfying an additional
holomorphicity condition at infinity for $g=1$).
More generally, if $g>1$ and $\rho: {\rm GL}_g \to {\rm GL}(V)$
denotes a finite-dimensional complex  representation of ${\rm GL}_g$,
then by $M_{\rho}(\Gamma_g)$ we mean the vector space of holomorphic
functions $f: \frak{H}_g \to V$ such that
$$
f((a\tau+b)(c\tau+d)^{-1})=\rho (c\tau+d) f(\tau)\, .
$$
If $\rho$ is an irreducible complex representation of ${\GL}_g$
of highest weight
$w=(a_1\geq a_2 \geq \cdots \geq a_g)$ then we denote
the complex vector space
$M_{\rho}(\Gamma_g)$ by $M_{\bf k}(\Gamma_g)$ with ${\bf k}=(a_1-a_2,a_2-a_3,\ldots, a_{g-1}-a_g,a_g)$.
The subspaces of cusp forms are denoted by $S_k(\Gamma_g)$ and
$S_{\rho}(\Gamma_g)$. These cusp forms can be interpreted as sections of
a line bundle or a vector bundle on the quotient space ${\mathcal A}_g=
\Gamma_g \backslash \frak{H}_g$.

If $f \in M_k(\Gamma_g)$ is a
scalar-valued modular form of weight $k$
 we can restrict $f$ to $\frak{H}_j \times\frak{H}_{g-j}$,
where we use the modular embedding
$ \Gamma_j \times \Gamma_{g-j} \to \Gamma_g$,
given by
$$
\left(\left( 
\begin{smallmatrix} a & b \\ c & d \\ \end{smallmatrix}  \right)
,
\left( 
\begin{smallmatrix} \alpha & \beta \\  \gamma & \delta \\ 
\end{smallmatrix}  \right) \right)
\mapsto 
\left( 
\begin{smallmatrix} 
a & 0 & b & 0 \\
0 & \alpha & 0 & \beta \\
c & 0 & d  & 0 \\
0 & \gamma & 0 & \delta \\
\end{smallmatrix}
\right)
$$
with corresponding map for the symmetric spaces
$$
\frak{H}_j \times \frak{H}_{g-j}\to \frak{H}_g,
\quad (\tau^{\prime}, \tau^{\prime\prime}) \mapsto
\left( 
\begin{matrix} \tau^{\prime} & 0 \\ 0 & \tau^{\prime \prime}\end{matrix}
\right) 
$$
with image ${\mathcal A}_{j,g-j}$ of ${\mathcal A}_j \times {\mathcal A}_{g-j}$ in ${\mathcal A}_g$.
The result is a tensor product of modular forms
$f^{\prime} \otimes f^{\prime \prime}
\in M_k(\Gamma_j)\otimes M_k(\Gamma_{g-j})$.
In \cite{Witt} Witt used this method to study
Siegel modular form of degree $2$.
However, often this restriction vanishes. If this is the case
we can develop $f$ in the conormal bundle of
$\frak{H}_j \times \frak{H}_{g-j}$ in $\frak{H}_g$.
This conormal bundle is a vector bundle of rank $j \, (g-j)$
with an action of
$\Gamma_j \times \Gamma_{g-j}$ and it descends to the tensor product
$$
N^{\vee}:={\EE}_j \boxtimes {\EE}_{g-j}:= 
p_j^*({\EE}_j) \otimes p_{g-j}^*({\EE}_{g-j})
$$
on ${\mathcal A}_j\times {\mathcal A}_{g-j}$ 
of the pullbacks
of Hodge bundles $\EE_j$ and $\EE_{g-j}$
on the factors ${\mathcal A}_j$ and ${\mathcal A}_{g-j}$.
Here $p_j$ (resp.\ $p_{g-j}$) denotes the projection of
$\Gamma_j \backslash \frak{H}_j \times \Gamma_{g-j} \backslash \frak{H}_{g-j}$
onto the factor $\Gamma_j \backslash \frak{H}_j$ (resp.\
$\Gamma_{g-j} \backslash \frak{H}_{g-j})$.
This readily can be checked by
a direct computation, but also follows from the observation that the
cotangent space to the moduli space
${\mathcal A}_g = \Gamma_g \backslash \frak{H}_g$ of principally polarized
complex abelian varieties
at a (general) 
point $X^{\prime}\times X^{\prime\prime}$  with $X^{\prime}$ (resp.\
$X^{\prime\prime}$) a principally polarized abelian variety
of dimension $j$ (resp.\ $g-j$), can be identified with
$$
{\rm Sym}^2(T_{X^{\prime}}^{\vee}) \oplus 
(T^{\vee}_{X^{\prime}}\otimes T^{\vee}_{X^{\prime\prime}}) 
\oplus  {\rm Sym}^2(T_{X^{\prime\prime}}^{\vee}),
$$
with $T_X=T_{X'}\oplus T_{X^{\prime\prime}}$
denoting the tangent space to $X$ at the origin,
and the middle term corresponds to the (co-)normal space.
Since for $g=2j$ the map of ${\mathcal A}_j \times {\mathcal A}_j \to
{\mathcal A}_g$ has degree $2$, we see the development of $f$
along $\mathfrak{H}_j \times \mathfrak{H}_j$ is symmetric, that is, 
invariant under the interchange of factors.

The Hodge bundle ${\EE}_g$ on ${\mathcal A}_g$ is associated to the
standard representation of ${\GL}_g$. 
In concrete terms it corresponds to the factor of automorphy $c\tau +d$. 
Its determinant is denoted $L_g$.
The bundle $N^{\vee}$ then corresponds to the tensor product of the 
standard representations of ${\GL}_j$ and ${\GL}_{g-j}$.

If $f$ is a scalar-valued modular form of weight $k$, that is,
a section of $L_g^k$, we can develop
$f$ in a point 
$(\tau^{\prime},\tau^{\prime\prime})\in \mathfrak{H}_j \times \mathfrak{H}_{g-j}$
in a Taylor series: if we write
$\tau \in \mathfrak{H}_g$ 
as $\left( \begin{smallmatrix} \tau' & z \\
z^t & \tau^{\prime\prime}\end{smallmatrix} \right)$
then we have
$$
f= t_0(f)+t_1(f)+t_2(f)+ \ldots \, , \eqno(1)
$$
where $t_r(f)$ is the sum of the terms of degree $r$ in the
coordinates $z_{a,b}$ with $1\leq a \leq j, 1\leq b \leq g-j$ of $z$;
in other words we develop $f$ along $\frak{H}_j \times 
\frak{H}_{g-j}$. If $t_0(f)$ vanishes then
we can interpret the term $t_1(f)$ as a section
of the conormal bundle tensored with $L_g^k$ 
over $\frak{H}_j \times \frak{H}_{g-j}$
and it descends to a section of the vector bundle
$L_j^k \otimes L_{g-j}^k \otimes N^{\vee}$
on ${\mathcal A}_j \times {\mathcal A}_{g-j}$, where
$L_j^k \otimes L_{g-j}^k$ is the restriction of the line
bundle $L_g^k$ on ${\mathcal A}_g$ 
whose sections are modular forms of weight $k$.
More generally, if $t_i(f)=0$ for $i=0,\ldots,r-1$ then $t_r(f)$
gives a section of $L_j^k\otimes L_{g-j}^k \otimes {\rm Sym}^r(N^{\vee})$.
In fact, the $n$th jet bundle $B_n$ tensored by $L_g^k$ admits a filtration
$$
L_g^k \otimes B_0 \subset L_g^k\otimes B_1 \subset \ldots \subset 
L_g^k \otimes B_n
$$
and the quotient $B_j/B_{j-1}$ is isomorphic to 
${\rm Sym}^j(N^{\vee})$.

We consider the $r$th symmetric power ${\rm Sym}^r(N^{\vee})$. 
Note that if we
have two vector spaces $V_1$ and $V_2$ of dimension $j$ and $g-j$ with
the standard
${\rm GL}_j$ and ${\rm GL}_{g-j}$-action then we have a
natural map
$$
{\rm Sym}^r(V_1\otimes V_2) \to {\rm Sym}^r(V_1) \otimes {\rm Sym}^r(V_2)
$$
which is an isomorphism if either $V_1$ or $V_2$ has dimension $1$.
This happens if $j=1$ and then 
${\rm Sym}^r(\EE_1 \otimes \EE_{g-1}) \cong {\EE}_1^r \otimes 
{\rm Sym}^{r}({\EE}_{g-1})$.
We begin by considering this case.

\begin{proposition}
Let $f \in M_k(\Gamma_g)$. If the restriction of $f$ along
$\frak{H}_1 \times\frak{H}_{g-1}$ vanishes to order $r$ and not to
higher order, then
the lowest order terms of $f$ along
$\frak{H}_1 \times \frak{H}_{g-1}$
define a non-vanishing element
$f'\otimes f^{\prime\prime} \in M_{k'}(\Gamma_1) \otimes M_{k^{\prime\prime}}(\Gamma_{g-1})$
where the weight $k'$ of $f'$ equals $k+r$ and that of $f^{\prime\prime}$
equals $k^{\prime\prime}=(r,0,\ldots,0,k)$.
\end{proposition}
\begin{proof} If $t_j(f)=0$ for $j<r$ but $t_r(f)\neq 0$ 
then $t_r(f)$ defines a non-zero
section of ${\rm Sym}^r(N^{\vee})\otimes L^k_{|{\mathcal A}_{1,g-1}}$
and the pullback of this to ${\mathcal A}_1\times {\mathcal A}_{g-1}$ 
is equal to ${\rm \EE}_1^{r+k} \otimes {\rm Sym}^r({\EE}_{g-1})\otimes L_{g-1}^k$.
\end{proof}

For given partition $g=j+(g-j)$ of $g$ we let ${\rm GL}_j \times {\rm GL}_{g-j}$
be the subgroup of ${\rm GL}_g$ that respects a decomposition $V=V_1\oplus V_2$.
In general, the space ${\Sym}^k(V_1\otimes V_2)$ admits a decomposition
as a direct sum of tensor products of irreducible representations of
${\rm GL}_j \times {\rm GL}_{g-j}$
in which the term ${\rm Sym}^{k}(V_1) \otimes {\rm Sym}^k(V_2)$
has high codimension.

If $f$ is a scalar-valued Siegel modular form of weight $k$
vanishing up to order $r$ along
$\frak{H}_j \times \frak{H}_{g-j}$ then we look at its term of order $r$.
This defines a section of
$$
L_j^k \boxtimes L_{g-j}^k  \otimes {\rm Sym}^r( \EE_j \boxtimes \EE_{g-j})
$$
and we can project it to each of the irreducible constituents of this bundle;
in particular we can consider the projection of the second factor to
${\rm Sym}^r(\EE_j) \boxtimes {\rm Sym}^{r}(\EE_{g-j})$.
The following proposition describes the result.

\begin{proposition}
If $f\in M_k(\Gamma_g)$ vanishes up to order $r$ along
$\frak{H}_j\times \frak{H}_{g-j}$ then the projection of its $r$th term
to $L_j^k\otimes L_{g-j}^k 
 \otimes {\rm Sym}^r(\EE_j) \otimes {\rm Sym}^{r}(\EE_{g-j})$ gives a
tensor product $f' \otimes f^{\prime\prime}$
with $f'\in M_{{k^{\prime}}}(\Gamma_j)$
(resp.\ $f^{\prime\prime} \in  M_{k^{\prime\prime}}(\Gamma_{g-j})$) with
$k^{\prime}$ (resp.\ $k^{\prime\prime}$) 
a vector of length $j$ (resp.\
 $g-j)$ of the form  $(r,0,\ldots,0,k)$ .
\end{proposition}

In general, if $V_n$ is the standard representation of ${\rm GL}_n$
and we write it as $V_n=V_j \oplus V_{g-j}$
then we have a decomposition of ${\rm Sym}^r(V_j\otimes V_{g-j})$ 
as a direct sum of isotypic subspaces for ${\GL}_j \times {\GL}_{g-j}$
$$
{\rm Sym}^r(V)= \oplus m_{\lambda',\lambda^{\prime\prime}} W_{\lambda'} \otimes W_{\lambda^{\prime\prime}} \eqno(2)
$$
If the restriction of a scalar-valued Siegel modular form $f$
vanishes up to order $r$ along $\frak{H}_j\times \frak{H}_{g-j}$,
then the projection onto any of the isotypic spaces in (2) gives a
vector-valued modular form. 

But we can also find modular forms 
further on in the Taylor expansion. 
The following proposition gives an example.

We write a form $f \in M_k(\Gamma_g)$ as 
a Taylor expansion $f=\sum_{n\geq 0} t_r(f)$ and decompose each term in its
isotypic components $t_n(f)=\sum_{\lambda} t_{n,\lambda}(f)$, where
$\lambda$ indexes the irreducible representations $R_{\lambda}$ 
that occur in
${\rm Sym}^n(V_j\otimes V_{g-j})$.

\begin{proposition}\label{variant3}
Let $f\in M_{k}(\Gamma_g)$ vanish to order $r$ on
$\frak{H}_j \times \frak{H}_{g-j}$ and write $f=\sum_{n\geq r} t_r(f)$.
Let $R_{\alpha}$ be an irreducible representation of ${\rm GL}_j\times
{\rm GL}_{g-j}$ occurring in the representation
${\rm Sym}^{r+1}(V_j\otimes V_{g-j})$ that does not occur in
$(V_j \otimes V_{g-j}) \otimes R_{\beta}$ for any $R_{\beta}$
for which $t_{r,\beta}(f)\neq 0$.
Then the projection $t_{r+1,\alpha}(f)$ of
$f$ to the $\alpha$-isotypic component of the $(r+1)$th step
is a modular form.
\end{proposition}
\begin{proof}
Exterior differentiation $d$ induces the homogeneous differential operator on 
$\frak{H}_j \times \frak{H}_{g-j}$ from
${\rm Sym}^r(N^{\vee})$ to ${\rm Sym}^{r+1}(N^{\vee})$.  
In terms of representations
of ${\rm GL}$ it corresponds to multiplication by $V_j \otimes V_{g-j}$.
In fact, the functional equation for $f$ says that
$f(\gamma \tau)=\rho(\gamma,\tau) f(\tau)$ for 
$\gamma \in \Gamma_j \times \Gamma_{g-j} \subset \Gamma_g$
and this implies 
$$
t_r(f(\gamma\tau))+t_{r+1}(f(\gamma\tau))=
t_0(\rho(\gamma,\tau))\, t_r(f(\tau))+
t_1(\rho(\gamma,\tau))\, t_r(f(\tau))+t_0(\rho(\gamma,\tau))\, t_{r+1}(f(\tau))
$$
and since $t_r(f)$ is a scalar-valued modular form we see that 
$$
t_{r+1}(f(\gamma\tau))=t_1(\rho(\gamma,\tau))\, t_r(f(\tau))+
t_0(\rho(\gamma,\tau))\, t_{r+1}(f(\tau))
$$
Note that $t_1(\rho(\gamma,\tau))$ is isomorphic to $V_j\otimes V_{g-j}$
as a ${\rm GL}$-representation.
Projection on the isotypic terms gives the result.
\end{proof}

\begin{question} 
In connection with the method presented here it is an interesting question
what the ideal is of scalar-valued Siegel modular forms of degree $g$
vanishing on $\frak{H}_j \times \frak{H}_{g-j}$. For $g=2$ and $j=1$ 
the answer is the ideal generated by $\chi_{10}$, but for other cases, 
even for $g=3$, the answer seems unknown. See also Igusa 
\cite{IgusaBAMS} and Sasaki \cite{Sasaki}.
\end{question}
\end{section}
\begin{section}{Example: Restricting The Schottky Form}
We illustrate the method by the example of the restriction of the 
Schottky form $J_8$, a cusp form of weight $8$ in degree $4$
along degree $2+2$
and degree $3+1$, cf.\ \cite{Igusa1}.
Its divisor in ${\mathcal A}_4$ is the closure of the locus of Jacobians.
The Schottky form can be described in various ways: 
as the difference of the theta series attached
to the two unimodular lattices $E_8\oplus E_8$ and $D_{16}^+$, but 
also as  the Ikeda lift of $\Delta \in S_{12}(\Gamma_1)$.
For yet other descriptions, we refer to \cite{PoorYuen} top of the page 209.
Since there are no cusp forms of weight $8$ on $\Gamma_2$, its restriction to
$\frak{H}_2\times \frak{H}_2$ vanishes. Using the action of the matrix
$$
\left(\begin{matrix}
a & 0 \\
0 & d \\
\end{matrix}
\right) 
\qquad \text{with $a=d=\left( \begin{matrix} 1_2 & 0 \\ 0 & -1_2\\
\end{matrix}\right) $}
$$
we see that $t_i(J_8)$ in the development (1) vanishes for $i$ odd.
If $V=V'\oplus V^{\prime\prime}$ is the standard representation of ${\GL}_4$
and $V'$ and $V^{\prime\prime}$ are the standard representations of ${\GL}_2$
then we have a decomposition
$$
{\rm Sym}^2(V'\otimes V^{\prime\prime})\cong
{\rm Sym}^2(V') \otimes {\rm Sym}^2(V^{\prime\prime}) \bigxor 
\wedge^2 V' \otimes 
\wedge^2 V^{\prime\prime}\, ,
$$
so starting with a modular form $f$ of weight $k$
that vanishes with order $2$ along
$\frak{H}_2\times \frak{H}_2$ the term $t_2(f)$ gives modular forms in
$$
M_{2,k}(\Gamma_2)\otimes M_{2,k}(\Gamma_2) \qquad \text {and} 
\qquad M_{0,k+1}(\Gamma_2)\otimes M_{0,k+1}(\Gamma_2)
$$
But in the case at hand $k=8$ and the spaces $S_{2,8}(\Gamma_2)$
and $S_{0,9}(\Gamma_2)$ vanish.
Therefore we pass to order $4$. Here we have the identification
of $\Sy^4(V'\otimes V^{\prime\prime})$ with
\begin{align*}
\Sy^4(V')\otimes \Sy^4(V^{\prime\prime}) \bigxor 
(\Sy^2(V')\otimes \wedge^2(V'))\otimes (\Sy^2(V^{\prime\prime})\otimes \wedge^2(V^{\prime\prime}))& \\ 
 \bigxor
 (\wedge^2(V')\otimes \wedge^2(V'))\otimes  
(\wedge^2(V^{\prime\prime})\otimes \wedge^2(V^{\prime\prime}))& \, .
\end{align*}
Thus for a modular form of weight $k$ on $\Gamma_4$
which vanishes at order four along $\frak{H}_2 \times \frak{H}_2$, 
$t_4(f)$ lies in
$$
\Sy^2 M_{4,k}(\Gamma_2)\oplus
\Sy^2M_{2,k+1}(\Gamma_2)\oplus
\Sy^2 M_{0,k+2}(\Gamma_2)
$$

The following Proposition shows that $\Delta$ and its siblings of degree 
$2$ and $3$ occur in the development of the Schottky form.

\begin{proposition}
For the Schottky form $J_8 \in S_8(\Gamma_4)$
of weight $8$ on $\Gamma_4$ the term $t_4(J_8)$ in the restriction from degree $4$
to degree $2+2$ is a non-zero multiple of
$$
\chi_{10}\otimes\chi_{10} \in 
\Sy^2 S_{10}(\Gamma_2)
$$
with $\chi_{10}$ a generator of $S_{10}(\Gamma_2)$, 
while the projection of $t_6(J_8)$ to $\Sy^2 S_{6,8}(\Gamma_2)$ is equal to
a non-zero multiple of
$$  
\chi_{6,8}\otimes \chi_{6,8} \in \Sy^2 S_{6,8}(\Gamma_2)
$$
where $\chi_{6,8}$ is a generator of $S_{6,8}(\Gamma_2)$.
The term $t_4(J_8)$ in the restriction from degree $4$ to degree $3+1$ is a
non-zero multiple of
$$
\chi_{4,0,8} \otimes \Delta
$$
with $\chi_{4,0,8}\in S_{4,0,8}(\Gamma_3)$ a generator.
\end{proposition}
\begin{proof} Since $S_{4,8}(\Gamma_2)=S_{2,9}(\Gamma_2)=(0)$ we
see that $t_4(J_8)$ lies in $\Sy^2 S_{10}(\Gamma_2)$. By calculating
the coefficients one sees that it does not vanish.
Indeed, if $(e_1,e_2)$ is a basis of $V_1$ and $(f_1,f_2)$ is a basis of $V_2$
we let $e_i\otimes f_j$ correspond to $\tau_{ij}$. 
The projection of ${\rm Sym}^4(V_1 \otimes V_2)$ on the direct summand
$(\wedge^2 V_1 \otimes \wedge^2 V_1 ) \otimes (\wedge^2 V_2 \otimes \wedge^2 V_2)$
 is then given by
$$
\frac{\partial^4}{\partial \tau_{13}^2\partial \tau_{24}^2}+\frac{\partial^4}{\partial \tau_{14}^2\partial \tau_{23}^2}
-2\frac{\partial^4}{\partial \tau_{13}\partial \tau_{14}\partial \tau_{23}\partial \tau_{24}}
$$
and a direct computation shows that
$$
t_4(J_8)=12\, \chi_{10}\otimes \chi_{10}.
$$

Playing the same game with the term $t_6(J_8)$ we find components in
$$
\Sy^6 (V)= \sum_{\lambda} V^{'}_{\lambda}\otimes V^{\prime\prime}_{\lambda}
$$
where $\lambda$ runs through $(6,0), (4,1), (2,2)$ and $(0,3)$.
By Proposition \ref{variant3}  
all components except the last one can give modular forms.
But the spaces $S_{4,9}(\Gamma_2)$, $S_{2,10}(\Gamma_2)$ and $S_{0,11}(\Gamma_2)$
are zero, therefore we find
$$
t_6(J_8)\in \Sy^2 S_{6,8}(\Gamma_2)
$$
and a calculation shows that this term does not vanish.
The argument for the restriction from degree $4$ to $3+1$ is similar.
\end{proof}
More details on the forms $\chi_{6,8}$ and the form in $S_{4,0,8}(\Gamma_3)$
can be found in sections \ref{TheFormZ} and \ref{TheFormF408}.
\end{section}
\begin{section}{Dimensions of Spaces of Scalar Cusp Forms}
It will be useful to have a table for the dimensions of scalar-valued
Siegel modular cusp forms of weight $4\leq k \leq 18$ and degree $1\leq g \leq 8$.
We will use this table as a heuristic tool that tells us where to look for
modular forms vanishing on $\frak{H}_j \times \frak{H}_{g-j}$. The tables of Ta\"{\i}bi
\cite{Taibi}
provide the dimensions in many cases, though some of his results are
conditional. For $g=2$ the dimensions were determined by Igusa, for $g=3$
by Tsuyumine. For degree $4$ and some higher degrees 
there are results of Poor and Yuen in  \cite{PoorYuen}.

\begin{footnotesize}
\smallskip
\vbox{
\bigskip\centerline{\def\quad{\hskip 0.6em\relax}
\def\quod{\hskip 0.5em\relax }
\vbox{\offinterlineskip
\hrule
\halign{&\vrule#&\strut\quod\hfil#\quad\cr
height2pt&\omit&&\omit && \omit&&\omit &&\omit &&\omit &&\omit && \omit&& \omit  &&\omit &&\omit &&\omit &&\omit &\cr
& $g\backslash k$ && $<$8 && 8 && 9 && 10 && 11 && 12 && 13 && 14 && 15 && 16 && 17 && 18 &\cr
\noalign{\hrule}
& 1 && 0 && 0 && 0 && 0 && 0 && 1 && 0 && 0 && 0 && 1 && 0 && 1 &\cr
& 2 && 0 && 0 && 0 && 1 && 0 && 1  && 0 && 1 && 0 && 2 && 0 && 2 &\cr
& 3 && 0 && 0 && 0 && 0 && 0 && 1  && 0 && 1 && 0 && 3 && 0 && 4 &\cr
& 4 && 0 && 1 && 0 && 1 && 0 && 2  && 0 && 3 && 0 && 7 && 0 && 12 &\cr
& 5 && 0 && 0 && 0 && 0 && 0 && 2  && 0 && 3 && 0 && 13 && 0 && 28 &\cr
& 6 && 0 && 0 && 0 && 1 && 0 && 3  && ? && 9 && 0 && 33 && 0 && 117 &\cr
& 7 && 0 && 0 && 0 && 0 && 0 && 3  && 0 && 9 && 0 && 83 && 0 && ? &\cr
& 8 && 0 && ? && 0 && $\geq 1$ && 0 && $\geq 4$   && $\geq 1$ && $\geq 23$ && $\geq 2$ && $\geq 234$  && ? && ? &\cr
} \hrule}
}}
\end{footnotesize}

\end{section}
\begin{section}{Restricting from Degree Three}
Restriction of forms on $\Gamma_2$ to forms on 
$\Gamma_1 \times \Gamma_1$ does not produce vector-valued 
modular forms, just scalar-valued ones,
though it shows how much the siblings of Delta are related: 
$\chi_{10} \in S_{10}(\Gamma_2)$ vanishes on 
$\frak{H}_1 \times \frak{H}_1$ and its first non-vanishing term, $t_2$, 
gives
$\Delta \otimes \Delta$, while in a similar way Igusa's
cusp form $\chi_{35} \in S_{35}(\Gamma_2)$ produces $\Delta^2e_{12}\otimes
\Delta^2 e_{12}$ with $e_{12}$ the Eisenstein series of weight $12$ 
on $\Gamma_1$.

The rings of scalar-valued Siegel modular forms on $\Gamma_g$  for $g\leq 3$
are  well-known.
(\cite{vdG1,Tsuyumine}).
A generator $\psi_{12}$ of  $S_{12}(\Gamma_3)$ is given by the following
combination of theta series associated to Niemeier lattices (of rank $24$)
$$
\psi_{12}=\frac{1}{1152} (-\vartheta_{\rm Leech}+6\, 
\vartheta_{24 A_1} -8\, \vartheta_{12 A_2}
+3\, \vartheta_{3A_8})\, ,
$$
where we refer for example to \cite{N-V} for notations for lattices.
The  coefficients in the
Fourier expansion $\sum_{N\geq 0} a(N) e^{2\pi i {\rm Tr}(N\tau)}$
of $\psi_{12}$ 
that we need are $a(1_3)=164$, $a(A_1(1/2))\oplus A_2(1/2))=18$ and
$a(A_3(1/2))=1$.
Thus the restriction of  $\psi_{12}$ to
$\frak{H}_1 \times \frak{H}_2$ does not vanish and equals
$24\, \Delta \otimes \chi_{12}$.

The restriction of a cusp form in $S_{14}(\Gamma_3)$ 
to $\frak{H}_1 \times \frak{H}_2$ must vanish.
Let $F$ be the  cusp form of weight $14$ that generates
$S_{14}(\Gamma_{3})$. The form
$F$ was constructed first by Miyawaki (\cite[p.\ 314--315]{Miyawaki})
and later by Ikeda as a lift (\cite{Ikeda}). We recall its construction.
Let

\[
D_{16}^+=
\left\{
x \in \QQ^{16} | \, 2x_i \in \ZZ, x_i-x_j \in \ZZ, x_1+\ldots +x_{16} \in 2\ZZ
\right\}
\]
be the unimodular lattice of rank 16 which is not $E_8 \oplus E_8$.
Let $Q$ be the $3\times 16$ complex matrix
$(1_3,\rho 1_3, \rho^2 1_3, 0)\in {\rm{Mat}}(3\times16,\CC)$,
with $\rho=e^{2 \pi i/3}$.  Then for a triple $(v_1,v_2,v_3)$
of vectors from $D_{16}^{+}$  we get a $3\times 3$ matrix $Q(v_1,v_2,v_3)$
and we define $F$ by
\[
F(\tau)=
\sum_{v_1,v_2,v_3 \in D_{16}^+}
{\rm{Re}}(\det(Q(v_1,v_2,v_3))^6) \,
e^{\pi i \sum_{i,j=1}^3 (v_i,v_j)\tau_{ij}}
\]
for
$
\tau=(\tau_{ij})\in \frak{H}_3
$.

We shall denote the Eisenstein series of weight $k$ for $\Gamma_1$ by $e_k$
and for $\Gamma_g$ with $g\geq 2$ by $E_k$.

\begin{proposition}
The lowest non-zero term in the development of $F$ along
$\frak{H}_2 \otimes \frak{H}_1$ is $t_2(F)$ and it is
a non-zero multiple of
the form $ \chi_{2,14}\otimes \Delta e_4$,  
where $\Delta e_4 \in S_{16}(\Gamma_1)$ and $\chi_{2,14}=
[\chi_{10},E_4]\in S_{2,14}(\Gamma_2)$.
\end{proposition}

\begin{proof}
We claim that the term $t_2(F)$ starts as follows
$$
t_2(F)= (2\pi i)^2
\left(
\begin{smallmatrix}
4q_{12}^{-1}+4+2x+4q_{12}\\
4(q_{12}-q_{12}^{-1})\\
4q_{12}^{-1}+4+2x+4q_{12}
\end{smallmatrix}
\right)
q_1q_2q_3+\cdots
$$
where $x$ is the Fourier coefficient of 
$\left(\begin{smallmatrix}
1 & 0 & 0\\ 0 & 1 & 1/2 \\ 0 & 1/2 & 1 \\ \end{smallmatrix}
\right)$ and where $q_j=e^{2 \pi i \tau_j}$ for $j=1,2,3$ and $q_{12}=e^{2 \pi i \tau_{12}}$.
This can be deduced from \cite{Miyawaki} (last table there) 
and the fact that the Fourier coefficients $a(N)$ of $F$ 
satisfy $a(N)=a(U^t N U)$ for $U \in {\rm GL}(3,{\ZZ})$.
We know that $t_2(F)$ lies in 
$S_{2,14}(\Gamma_2) \otimes S_{16}(\Gamma_1)$
and does not vanish. 
We can construct a generator of $S_{2,14}(\Gamma_2)$ 
by the bracket construction, cf.\ \cite{Ibukiyama,C-vdG-G}.
The Fourier expansions of $\chi_{10}$  and the Eisenstein series
$E_4\in M_4(\Gamma_2)$ start as follows:
$$
\chi_{10} (\tau)
=(q_{12}^{-1}-2+q_{12})\, q_1q_2+\cdots
\quad {\rm and} \quad
E_4 (\tau)=
1+240\, (q_1+q_2) +\cdots
$$
so we have
$$
[\chi_{10},E_4]
=
10
\left(
\begin{smallmatrix}
q_{12}^{-1}-2+q_{12}\\
q_{12}-q_{12}^{-1}\\
q_{12}^{-1}-2+q_{12}
\end{smallmatrix}
\right)
q_1q_2+\cdots 
$$
The Fourier expansion of $\Delta e_4$ starts by $\Delta e_4 (\tau_3)=q_3+\ldots$ so we get
\[
[\chi_{10},E_4]
\left(
\begin{smallmatrix}
\tau_1 &  \tau_{12}\\
 \tau_{12} & \tau_2\\
\end{smallmatrix}
\right)
\otimes 
\Delta e_4 (\tau_3)
=
10
\left(
\begin{smallmatrix}
q_{12}^{-1}-2+q_{12}\\
q_{12}-q_{12}^{-1}\\
q_{12}^{-1}-2+q_{12}
\end{smallmatrix}
\right)
q_1q_2q_3+\ldots 
\]
It follows that $x=-6=a
(\left(
\begin{smallmatrix}
1 &  0  & 0 \\
0 & 1 & 1/2\\
0 & 1/2 & 1
\end{smallmatrix}
\right))$ for the unknown Fourier coefficient of $F$.
\end{proof}
In weight $18$ there is a well-known scalar-valued cusp form $\chi_{18}$ 
of degree~$3$
that vanishes along the locus of Jacobians of hyperelliptic curves of
degree~$3$. It is defined as the product of the $36$ even theta
characteristics in degree~$3$. The Fourier expansion of this $\chi_{18}$ 
starts as follows
$$
2^{28} \, (108-60\, (q_{12}^{-1}+q_{12}+q_{13}^{-1}+q_{13}+q_{23}^{-1}+q_{23}+\cdots)
\, q_1^2q_2^2q_3^2 + \cdots)
$$
showing that $\chi_{18}$ vanishes of order $2$ at infinity.
The restriction to $\frak{H}_2 \times \frak{H}_1$ vanishes because
this restriction lies in $S_{18}(\Gamma_2) \otimes S_{18}(\Gamma_1)$,
and because $\chi_{18}$ vanishes twice at infinity the components of its
restriction do so too,
and there is no cusp form of weight $< 24$ on $\Gamma_1$
vanishing twice at infinity. 

\begin{proposition} Along $\frak{H}_2 \times \frak{H}_1$
we have
$$
t_6(\chi_{18})=
c \, \chi_{10} \,  \chi_{6,8} \otimes \Delta^2  \, \in \, S_{6,18}(\Gamma_2)
\otimes S_{24}(\Gamma_1)\, ,
$$
where $c\neq 0$ 
and $\chi_{6,8}$ is a generator of $S_{6,8}(\Gamma_2)$.
\end{proposition}
\begin{proof} From the fact that there is no cusp form
of weight less than $24$ on $\Gamma_1$ that vanishes twice at the cusp 
it follows that $t_i(\chi_{18})$ 
vanishes for $i\leq 5$.
We know that the $t_6(\chi_{18})$ lies in
$S_{6,18}(\Gamma_2)\otimes S_{24}(\Gamma_1)$.
But the subspace of $S_{24}(\Gamma_1)$ of elements vanishing twice at
infinity is generated by $\Delta^2$. Moreover, the calculation
\[
\left(
\begin{smallmatrix}
\frac{\partial^6 \chi_{18}}{\partial \tau_{13}^6}\\
6\frac{\partial^6 \chi_{18}}{\partial \tau_{13}^5\partial\tau_{23}}\\
\vdots\\
6\frac{\partial^6 \chi_{18}}{\partial \tau_{13}\partial\tau_{23}^5}\\
\frac{\partial^6 \chi_{18}}{\partial \tau_{23}^6}
\end{smallmatrix}
\right)
\left(
\begin{smallmatrix}
\tau_1 &  \tau_{12}  & 0 \\
 \tau_{12} & \tau_2 & 0\\
0 & 0 & \tau_{3}
\end{smallmatrix}
\right)
=
\left(
\begin{smallmatrix}
0 \\
0 \\
q_{12}^{-2}-4q_{12}^{-1}+6-4q_{12}+q_{12}^2\\
-2q_{12}^{-2}+4q_{12}^{-1}-4q_{12}+2q_{12}^2\\
q_{12}^{-2}-4q_{12}^{-1}+6-4q_{12}+q_{12}^2\\
0 \\
0 \\
\end{smallmatrix}
\right)
q_1^2q_2^2q_3^2+\ldots 
\]
shows that it does not vanish and is divisible by $\chi_{10}$ because
substitution of $q_{12}=1$ gives zero. Since $\dim S_{6,8}(\Gamma_2)=1$
the result follows.
\end{proof}
\end{section}
\begin{section}{Restricting from Degree $4$}
We begin by listing the modular forms that we are going to restrict.
As before, we denote by $e_{k}$ the Eisenstein series of weight $k$ 
on $\Gamma_1$ and by $E_{k}$ the Eisenstein series of weight $k$ 
in higher genera, always normalized such that their
Fourier expansion starts with~$1$.

Besides the Schottky form $J_8$ that generates $S_8(\Gamma_4)$ we
have the generator
$F_{10}= -I_4(e_4 \Delta)/168  \, \in S_{10}(\Gamma_4)$
where $I_4$ is the Ikeda lift $I_4:S_{16}(\Gamma_1) \to S_{10}(\Gamma_4)$,
the two generators
$ G_1=I_4(e_4^2 \Delta)/360 \quad \text{and} \quad G_2= -J_8 E_4/2 $
of $S_{12}(\Gamma_4)$,
the three Hecke eigenforms $H_1,H_2,H_3$ that generate $S_{14}(\Gamma_4)$,
see \cite[p.\ 213]{PoorYuen} (but note that the eigenvalues given there 
are not correct; in fact, the expressions of $f_7$ and $f_8$ on page 214
and the eigenvalues for $J_8$ in table 2 on page 218 are incorrect),
where
$$
H_1=I_4(\Delta e_4^3- (156-12 \alpha) \Delta^2), \quad \text{and} \quad 
H_2=I_4(\Delta e_4^3- (156+12 \alpha) \Delta^2)
$$
with $\alpha=\sqrt{144169}$. 
One can calculate the first Fourier coefficients. 
We give the results in a table.
The coefficients of $H_2$ are the conjugates of those of $H_1$. The space 
$S_{14}(\Gamma_4)$ contains $E_6 J_8$.

    Note that the paper \cite{Keaton} gives a closed formula for the
eigenvalues of an Ikeda lift.
\begin{footnotesize}
\smallskip
\vbox{
\bigskip\centerline{\def\quad{\hskip 0.6em\relax}
\def\quod{\hskip 0.5em\relax }
\vbox{\offinterlineskip
\hrule
\halign{&\vrule#&\strut\quod\hfil#\quad\cr
height2pt&\omit&&\omit &&\omit &&\omit &&\omit &&\omit &&\omit  &\cr
& &&
$J_8$ && $F_{10}$ && $G_{1}$ && $G_{2}$ && $H_{1}$ && $H_{3}$  &\cr
\noalign{\hrule}
& $1_4$ &&
$40$ && $472$  && $-4440$ && $-40$ && $-434984-968\alpha$   && $-2080$  &\cr
& $2A_1(1/2)\oplus A_2(1/2)$ &&
$-12$ && $-36$  && $-492$ && $12$ && $63132+204\alpha$  && $288$  &\cr
& $2A_2(1/2)$ &&
$6$ && $72$  && $-78$ && $-6$ && $-44904-48\alpha$ && $-198$  &\cr
& $A_1(1/2)\oplus A_3(1/2)$ &&
$2$ && $-22$  && $-38$ && $-2$ && $4994-22\alpha$ && $28$  &\cr
& $A_4(1/2)$ &&
$-1$ && $2$  && $1$ && $1$ && $-274+2\alpha$ && $-5$  &\cr
& $D_4(1/2)$ &&
$1$ && $1$  && $-3$ && $-1$ && $-467+\alpha$ && $2$  &\cr
} \hrule}
}}
\end{footnotesize}

According to \cite{B-F-vdG} we should
have $\dim S_{4,0,8}(\Gamma_3)=1$,
$\dim S_{2,0,10}(\Gamma_3)=1$ and $\dim S_{2,0,14}(\Gamma_3)=2$. 
We denote the generating eigenforms by $\chi_{4,0,8}$, $\chi_{2,0,10}$,
and $\chi_{2,0,14}$ and $\chi_{2,0,14}^{\prime}$. The form
$\chi_{2,0,14}$ and its conjugate $\chi_{2,0,14}^{\prime}$ are lifts
with Hecke eigenvalues of the form
$$
a(p)(p^{11}+b(p)+p^{12})
$$
with $a(p)$ the eigenvalue of the eigenform of $S_{16}(\Gamma_1)$
and $b(p)$ the eigenvalue of a Hecke eigenform in $S_{24}(\Gamma_1)$.

\begin{proposition}
By restricting scalar-valued modular cusp forms of degree $4$ and small
weight to $\frak{H}_3 \times \frak{H}_1$ 
we find (a non-zero multiple of) 
the vector-valued modular forms as in the table below where
$\chi_{0,0,12}=\psi_{12}$.
\end{proposition}
\begin{footnotesize}
\smallskip
\vbox{
\bigskip\centerline{\def\quad{\hskip 0.6em\relax}
\def\quod{\hskip 0.5em\relax }
\vbox{\offinterlineskip
\hrule
\halign{&\vrule#&\strut\quod\hfil#\quad\cr
height2pt&\omit &&\omit &&\omit &&\omit &&\omit  &\cr
& $k$ && $\dim S_k(\Gamma_4)$ && form  && $r$ && $t_r$   &\cr
\noalign{\hrule}
& $8$ && $1$  && $J_8$ && $4$ && $\chi_{4,0,8}\otimes \Delta$ &\cr
& $10$ && $1$  && $F_{10}$ && $2$ && $\chi_{2,0,10}\otimes \Delta$  &\cr
& $12$ && $2$  && $G_1$ && $0$ && $\chi_{0,0,12}\otimes \Delta$ &\cr
&  &&   && $G_2$ && $2$ && $E_4 \chi_{6,8} \otimes e_4 \Delta$ &\cr
& $14$ && $3$  && $H_1$ && $2$ && $\chi_{2,0,14} \otimes e_4\Delta$ &\cr
& &&   && $H_2$ && $2$ &&  $\chi_{2,0,14}^{\prime} \otimes e_4\Delta$ &\cr
&  &&   && $E_6J_8$ && $4$ && $E_6 \chi_{4,0,8} \otimes e_6 \Delta$ &\cr
} \hrule}
}}
\end{footnotesize}

\begin{proposition}
By restricting scalar-valued modular cusp forms of degree $4$ and small
weight to $\frak{H}_2\times \frak{H}_2$
we find the vector-valued modular forms as in the table below.
\end{proposition}
\begin{footnotesize}
\smallskip
\vbox{
\bigskip\centerline{\def\quad{\hskip 0.6em\relax}
\def\quod{\hskip 0.5em\relax }
\vbox{\offinterlineskip
\hrule
\halign{&\vrule#&\strut\quod\hfil#\quad\cr
height2pt&\omit &&\omit &&\omit &&\omit &&\omit  &\cr
& $k$ && $\dim S_k(\Gamma_4)$ && form  && $r$ && $t_r$   &\cr
\noalign{\hrule}
& $8$ && $1$  && $J_8$ && $6$&& $\chi_{6,8} \otimes \chi_{6,8}$ &\cr
& $10$ && $1$  && $F_{10}$ && $0$&& $\chi_{10}\otimes \chi_{10}$  &\cr
& $12$ && $2$  && $G_1$ && $0$&& $\chi_{12}\otimes \chi_{12}$ &\cr
&  &&   && $G_2$ && $6$&& $E_4 \chi_{6,8}\otimes E_4 \chi_{6,8}$ &\cr
& $14$ && $3$  && $H_1$ && $0$&& $E_4\chi_{10} \otimes E_4 \chi_{10}$ &\cr
& &&   && $H_1-H_2$ && $2$&&  $\chi_{2,14}\otimes \chi_{2,14}$ &\cr
&  &&   && $E_6J_8$ && $6$&& $E_6\chi_{6,8}\otimes E_6 \chi_{6,8}+ E_6\chi_{10}\otimes E_6 \chi_{10}$ &\cr
} \hrule}
}}
\end{footnotesize}

Note that if $f$ is a form in $S_k(\Gamma_4)$ 
such that along $\frak{H}_2\times \frak{H}_2$ we have 
that $t_r(f)=0$ for $r<6$ then $t_6(f)$ lies in 
$$
{\rm Sym}^2(S_{6,k}(\Gamma_2)) \oplus  {\rm Sym}^2(S_{4,k+1}(\Gamma_2)) 
\oplus {\rm Sym}^{2}(S_{2, k+2}(\Gamma_2)) \oplus
{\rm Sym}^2(S_{0,k+3}(\Gamma_2))\, .
$$

We can deduce the Fourier expansion of these forms. For example, 
for the Fourier expansion of $\chi_{2,0,10}/12 (2\pi i)^2$ 
is given by

$$
\big[
\left(
\begin{smallmatrix}
-20 \\ 0 \\ 0 \\ -20 \\ 0 \\ -20
\end{smallmatrix}
\right)
+
\left(
\begin{smallmatrix}
-6 \\ -6 \\ 0 \\ -6 \\ 0 \\ 14
\end{smallmatrix}
\right)
q_{12}+
\left(
\begin{smallmatrix}
-6 \\ 6 \\ 0 \\ -6 \\ 0 \\ 14
\end{smallmatrix}
\right)
q_{12}^{-1}
+
\left(
\begin{smallmatrix}
-6 \\ 0 \\ -6 \\ 14 \\ 0 \\ -6
\end{smallmatrix}
\right)
q_{13}
+\ldots+
\left(
\begin{smallmatrix}
1 \\ 1 \\ 1 \\ 1 \\ 0 \\ 1
\end{smallmatrix}
\right)
q_{12}q_{13}+
\ldots
\big]
q_1q_2q_3
+\ldots
$$
and that of $\chi_{2,0,14}/12 (2\pi i)^2)$ is given by

$$
\big[
\left(
\begin{smallmatrix}
13540+20\alpha \\ 0 \\ 0 \\ 13540+20\alpha \\ 0 \\ 13540+20\alpha
\end{smallmatrix}
\right)
+
\left( 
\begin{smallmatrix}
-6\,\alpha+1482\\ 
-6\,\alpha+1482\\ 
0\\ 
-6\,\alpha+1482\\ 
0\\ 
-6\,\alpha-7758
\end {smallmatrix}
\right)
q_{12}
+\ldots+
\left(
\begin{smallmatrix}
\alpha -247 \\ \alpha -247 \\ \alpha -247 \\ \alpha -247 \\ 0 \\ \alpha -247
\end{smallmatrix}
\right)
q_{12}q_{13}+
\ldots
\big] \,
q_1q_2q_3
+\ldots
$$
\end{section}

\begin{section}{Restriction from Degree Six and Eight}
We begin by constructing a form of weight $10$ both in degree $8$ and degree $6$.
The Ikeda lift of $\Delta \in S_{12}(\Gamma_1)$ to degree $8$ gives a form $I_8(\Delta)$ in $S_{10}(\Gamma_8)$. For our purpose we need a number of Fourier coefficients
of $I_8(\Delta)=\sum_T b(T) e^{2 \pi i {\rm tr}(T\cdot \tau)}$; 
in fact, we need these for all positive definite half-integral
symmetric matrices with diagonal equal to $(1,\ldots,1)$. For a
positive definite half-integral
symmetric matrix $T$ with fundamental discriminant $D_T$ the Fourier coefficient
is given by $b(T)=c(|D_T|)=c(\det (2T))$, where 
$$
h=\sum_{n\geq 1, n\equiv 0,1 \, \bmod 4}  c(n) q^n \in S_{13/2}^{+}(\Gamma_0(4))
$$
is the form of half-integral weight corresponding to $\Delta$ under the Shimura
correspondence. By restricting this form $I_8(\Delta)$ to $\frak{H}_7\times \frak{H}_1$ we find linear relations between the Fourier coefficients and this gives a way
of calculating further Fourier coefficients. Indeed, the restriction 
$r_{7,1}(I_8(\Delta))$ is zero
in view of dim $S_{10}(\Gamma_7)=\dim S_{10}(\Gamma_1)=0$ and 
we thus find that the Fourier coefficient at $N\otimes A_1(1/2))$ of 
$r_{7,1}(I_8(\Delta))$ is given by 
$$
\sum_{N\oplus A_1(1/2)=T} b(I_8(\Delta),T),
$$
where the sum runs over all positive definite $T$ with upper left block $N$ and lower
right block $A_1(1/2)$. This gives a relation $b(I_8(\Delta),A_7(1/2)\oplus A_1(1/2)
) + 56 b(I_8(\Delta),E_8(1/2))=0$. In this way we obtain a number of relations
between the Fourier coefficients.

Furthermore we consider the restriction of $I_8(\Delta)$ to $\frak{H}_6\times 
\frak{H}_2$. We know that $S_{10}(\Gamma_2)$ is generated by the form $\chi_{10}$
and that $\dim S_{10}(\Gamma_6)=1$. Then restricting further to 
$\frak{H}_5\times \frak{H}_1 \times \frak{H}_2$, 
$\frak{H}_4\times \frak{H}_2 \times \frak{H}_2$ and 
$\frak{H}_3\times \frak{H}_3 \times \frak{H}_2$ 
gives further relations. Together these suffice to determine all the coefficients
$b(I_8(\Delta),T)$ for $T$ a positive definite half-integral
symmetric matrix with diagonal equal to $(1,\ldots,1)$.
At the same time it gives us a number of Fourier coefficients of the generator
$G \in S_{10}(\Gamma_6)$. This shows that we get a non-zero cusp 
form of weight $10$ on $\Gamma_6$.
The results are given in two tables.

\begin{footnotesize}
\smallskip
\vbox{
\bigskip\centerline{\def\quad{\hskip 0.6em\relax}
\def\quod{\hskip 0.5em\relax }
\vbox{\offinterlineskip
\hrule
\halign{&\vrule#&\strut\quod\hfil#\quad\cr
height2pt&\omit &&\omit &&\omit &&\omit &&\omit && \omit &\cr
& $2\, T_i$ && $b(I_8(\Delta),T_i)$ && $2\, T_i$ && $b(I_8(\Delta),T_i)$ && $2\, T_i$ && $b(I_8(\Delta),T_i)$ & \cr
\noalign{\hrule}
& $8A_1$ && $146657280$ && $6A_1\oplus A_2$ && $9676800$ && $4A_1\oplus 2A_2$ && $3456000$ & \cr
& $5A_1\oplus A_3$ && $-1612800$ && $2A_1\oplus 3A_2$ && $362880$ && $3A_1\oplus A_2\oplus A_3$ && $311040$ &\cr
& $4A_2$ && $1970568$ && $4A_1\oplus A_4$ && $-760320$ && $A_1\oplus 2A_2\oplus A_3$ && $-293760$ & \cr
& $2A_1\oplus 2A_3$ && $393728$ && $4A_1\oplus D_4$ && $-523776$ && $2A_1\oplus A_2 \oplus A_4$ && $-51840$ & \cr
& $A_2\oplus 2A_3$ && $-126720$ && $2A_1\oplus A_2\oplus A_4$ && $-34560$ && $3A_1\oplus A_5$ && $-34560$ & \cr
& $2A_2\oplus A_4$ && $146880$ && $A_1\oplus A_3 \oplus A_4$ && $23520$ && $2A_2\oplus D_4$ && $86400$ & \cr
& $A_1\oplus A_2 \oplus A_5$ && $-41328$ && $A_1\oplus A_3 \oplus D_4$ && $5760$ && $3A_1\oplus D_5 $ && $5760$ & \cr
& $2A_2\oplus A_6$ && $13440$ && $2A_4$ && $17330$ && $A_3\oplus D_5$ && $-12960$ & \cr
& $A_1\oplus A_2 \oplus D_5$ && $-12960$ && $A_2\oplus A_6$ && $5040$ && $A_4\oplus D_4$ && $8640$ & \cr
& $2D_4$ && $4416$ && $A_3\oplus D_5$ && $-4288$ && $A_1\oplus A_7$ && $-704$ & \cr
& $2A_1\oplus D_6$ && $3392$ && $A_2\oplus D_6$ && $1440$ && $2A_1\oplus E_6$ && $1440$ & \cr
& $A_2\oplus E_6$ && $1440$ && $A_8$ && $9$ && $A_1\oplus D_7$ && $-240$ & \cr
& $D_8$ && $8$ && $A_1\oplus E_7$ && $-56$ && $E_8$ && $1$ & \cr
} \hrule}
}}
\end{footnotesize}

\begin{footnotesize}
\smallskip
\vbox{
\bigskip\centerline{\def\quad{\hskip 0.6em\relax}
\def\quod{\hskip 0.5em\relax }
\vbox{\offinterlineskip
\hrule
\halign{&\vrule#&\strut\quod\hfil#\quad\cr
height2pt&\omit && \omit && \omit && \omit&& \omit && \omit && \omit && \omit &\cr
& $2N$ && $a(G,N)$ && $2N$ && $a(G,N)$ && $2N$ && $a(G,N)$ && $2N$ && $a(G,N)$ & \cr
\noalign{\hrule}
& $6A_1$ && $−280320$ && $4A_1\oplus A_2$ && $-15744$ && $2A_1\oplus 2A_2$ && 
$−8496$ && $3A_1\oplus A_3$ && $10176$ & \cr  \cr
& $3A_2$ && $12996$ && $A_1\oplus A_2 \oplus A_3$ && $−2472$ && $2A_1\oplus A_4$ && $1000$ && $2A_1\oplus D_4$ && $384$ & \cr
& $2A_3$ && $-1040$ && $A_2\oplus A_4$ && $750$ && $A_1\oplus A_5$ && $-164$ && $A_2\oplus D_4$ && $384$ & \cr
& $A_1\oplus D_5$ && $-52$ && $A_6$ && $7$ && $D_6$ && $2$ && $E_6$ && $1$ & \cr
} \hrule}
}}
\end{footnotesize}

\begin{proposition} The restriction of second order of $I_8(\Delta)$ along 
$\frak{H}_7 \times \frak{H}_1$ is a non-zero form in $S_{2,0,0,0,0,0,10}(\Gamma_7) \otimes S_{12}(\Gamma_1)$.
The restriction of the generator $G \in S_{10}(\Gamma_6)$ along $\mathfrak{H}_5 \times \mathfrak{H}_1$ is a non-zero form
in $S_{2,0,0,0,10}(\Gamma_5) \otimes S_{12}(\Gamma_1)$.
\end{proposition}
\begin{proof}
We write the term $t_2(I_8(\Delta))$ as $F\otimes \Delta$ with $F \in S_{2,0,0,0,0,0,10}(\Gamma_7)$. It is given as the transpose of
$$
(
\frac{\partial^2 I_8(\Delta)}{\partial \tau_{18}^2},
2\frac{\partial^2 I_8(\Delta)}{\partial \tau_{18}\tau_{28}},
\ldots,
\frac{\partial^2 I_8(\Delta)}{\partial \tau_{78}^2}
)
\vert_{{\mathfrak{H}_7 \times \mathfrak{H}_1}}
$$
Then the Fourier expansion of $F$ starts as follows:
$$
F(\tau)=
\sum_{N>0} a(N) e^{2 \pi i {\rm{Tr}}(N\tau)}\\
=
\sum_{N\in \mathcal{I}_7} P_N(q_{ab},q_{ab}^{-1})\cdot q_1\cdots q_7+\ldots
$$
where
\[
\tau= (\tau_{ij})
\in \mathfrak{H}_7,
\quad
q_a=e^{2 \pi i \tau_{aa}},
\quad 
q_{ab}=e^{2 \pi i \tau_{ab}},
\quad 
P_N(X,X^{-1})\in \CC[X,X^{-1}]^{28}
\]
and $\mathcal{I}_7$ is the set of symmetric positive definite half-integral
$7 \times 7$ matrices with $(1,\ldots,1)$ on the diagonal.
The `constant' term in the Fourier expansion of $F$ is $a(1_7)$ and we get
it by substituting $\tau_{18}=\ldots = \tau_{78}=0$ in the transpose of 
$$
(
\frac{\partial^2 P}{\partial \tau_{18}^2},
2\frac{\partial^2 P}{\partial \tau_{18}\tau_{28}},
\ldots,
\frac{\partial^2 P}{\partial \tau_{78}^2}
)
$$
where
\[
P=\sum_{N\in \mathcal{I}_8} 
b(I_8(\Delta),N)
q_{18}^{n_{18}}\cdots q_{78}^{n_{78}}
\]
and $\mathcal{I}_8$ is the set of positive definite half-integral matrices
of the form
$
(\begin{matrix} 
1_7 & n \\ n^t & 1_1 \\ 
\end{matrix}
)
$.
This set of matrices contains $379$ elements and we can classify them modulo 
${\rm GL}(7,{\ZZ})$-equivalence.  As it turns out the lattices $8A_1(1/2)$,
$6A_1(1/2)\oplus A_2(1/2)$, $5A_1(1/2)\oplus A_3(1/2)$ and 
$4A_1(1/2)\oplus D_4(1/2)$
occur with multiplicities $1$, $14$, $84$ and $280$. 
We thus find
$$
\begin{aligned}
P=&\, b(I_8(\Delta),8\, A_1(1/2))+
b(I_8(\Delta),6\, A_1(1/2)\oplus A_2(1/2))(q_{18}+q_{18}^{-1}+\ldots+q_{78}+q_{78}^{-1})\\
&+
b(I_8(\Delta),5\, A_1(1/2)\oplus A_3(1/2))(q_{18}q_{28}+\ldots+q_{18}q_{78}+\ldots)\\
&+
b(I_8(\Delta),4\, A_1(1/2)\oplus D_4(1/2))(q_{18}q_{28}q_{38}+\ldots+q_{18}q_{68}q_{78}+\ldots)\\
=&\,
146657280+ 
9676800(q_{18}+q_{18}^{-1}+\ldots+q_{78}+q_{78}^{-1})\\
&
-1612800(q_{18}q_{28}+\ldots+q_{18}q_{78}+\ldots)
-523776(q_{18}q_{28}q_{38}+\ldots+q_{18}q_{68}q_{78}+\ldots).
\end{aligned}
$$
Thus we get
\[
^t a(1_7)=4\pi^2\cdot 82206720\cdot
[1, 0, 0, 0, 0, 0, 0, 1, 0, 0, 0, 0, 0, 1, 0, 0, 0, 0, 1, 0, 0, 0, 1, 0, 0, 1, 0, 1]
\]
which shows our result for $I_8(\Delta)$. For the restriction of $G$ the argument is similar. Instead of $\mathcal{I}_8$
we have a set $\mathcal{I}_6$ of $131$ elements representing the lattices 
$6A_1(1/2)$, $4A_1(1/2)\oplus A_2(1/2)$, $3A_1(1/2)\oplus A_3(1/2)$ and $2A_1(1/2)\oplus D_4(1/2)$ with multiplicities
$1$, $10$, $40$ and $80$ from which we find 
for the constant term $a(1_5)$ the transpose of
$$
-4\pi^2 149760 \, [1,0,0,0,0,1,0,0,0,1,0,0,1,0,1]
$$
\end{proof}

\smallskip

Along similar lines one finds the following result.

\begin{proposition}
The second order restriction $t_2(G)$ of the generator $G\in S_{10}(\Gamma_6)$  
along $\mathfrak{H}_5\times \mathfrak{H}_1$
is a non-zero form in $S_{2,0,0,0,10}(\Gamma_5)\otimes S_{12}(\Gamma_1)$. Furthermore,
the second order restriction of $G$ along $\mathfrak{H}_3\times \mathfrak{H}_3$
yields a form $f\otimes f$ with $f$ a non-zero form in $S_{2,0,10}(\Gamma_3)$.
\end{proposition}

Note that the spaces $S_{2,0,0,0,10}(\Gamma_5)$ and $S_{2,0,10}(\Gamma_3)$ are both $1$-dimensional according to Ta\"ibi \cite{Taibi}.

\end{section}
\begin{section}{The Sibling $\chi_{6,8}$ of degree $2$}\label{TheFormZ}
As we have seen the form $\chi_{6,8}$ in $S_{6,8}(\Gamma_2)$ 
appears ubiquitously. Its presence was first seen in the 
cohomology of local systems on the moduli spaces ${\mathcal A}_2$ 
and ${\mathcal M}_2$ in \cite{F-vdG}. One of us asked Ibukiyama
whether he could construct a form in $S_{6,8}(\Gamma_2)$.
Ibukiyama answered in 2001 with a construction of this form using theta
functions with pluriharmonic polynomial coefficients.
Now we have easier ways to construct it.
One way is as follows.
Let $G_i^t=(\partial \vartheta_i/\partial z_1, \partial \vartheta_i /\partial
z_2)$ be the (transposed) gradient of the $i$th odd theta function for
$i=1,\ldots,6$, see \cite{C-vdG-G}. 
It defines a section of ${\EE}_2\otimes \det({\EE}_2)^{1/2}$ 
for the congruence subgroup $\Gamma_2[4,8]$. We let ${\rm Sym}^j({\EE}_2)$
be the $\frak{S}_j$-invariant subbundle of ${\EE}_2^{\otimes j}$. 
Then the expression
$$
{\rm Sym}^j(G_1,\ldots,G_6)
$$
defines a cusp form $f_{6,3}$ of weight $(6,3)$ on the principal 
congruence subgroup $\Gamma_2[2]$. The product  
$\chi_5  \, f_{6,3}$ with $\chi_5$, the product of the ten even 
theta characteristics, is a form of level $1$ and
is up to a normalization equal to $\chi_{6,8}$: 
$$
\chi_{6,8}:=-(\chi_5\Sy^6(G_1,\ldots, G_6))/4096 \pi^{6}
$$
We write its Fourier expansion as
$$
{\chi}_{6,8}(\tau)=\sum_{N > 0}
a(N) e^{2 \pi i {\rm{Tr}}(N\tau)}
=
\sum_{N > 0}
\, ^t(a(N)_0,\ldots,a(N)_6) e^{2 \pi i {\rm{Tr}}(N\tau)} \, .
$$
It starts as follows (with $q_1=e^{2\pi i \tau_1}$, 
$q_2=e^{2\pi i \tau_2}$ and $r=e^{2\pi i \tau_{12}}$)

\begin{align*}
\chi_{6,8}(\tau)=&
\left(
\begin{smallmatrix}
0\\
0\\
r^{-1}-2+r\\
2(r-r^{-1})\\
r^{-1}-2+r\\
0\\
0
\end{smallmatrix}
\right)
q_1q_2+
\left(
\begin{smallmatrix}
0\\
0\\
-2(r^{-2}+8r^{-1}-18+8r+r^2)\\
8(r^{-2}+4r^{-1}-4r-r^2)\\
-2(7r^{-2}-4r^{-1}-6-4r+7r^2)\\
12(r^{-2}-2r^{-1}+2r-r^2)\\
-4(r^{-2}-4r^{-1}+6-4r+r^2)
\end{smallmatrix}
\right)
q_1q_2^2\\
&+
\left(
\begin{smallmatrix}
-4(r^{-2}-4r^{-1}+6-4r+r^2)\\
12(r^{-2}-2r^{-1}+2r-r^2)\\
-2(7r^{-2}-4r^{-1}-6-4r+7r^2)\\
8(r^{-2}+4r^{-1}-4r-r^2)\\
-2(r^{-2}+8r^{-1}-18+8r+r^2)\\
0\\
0
\end{smallmatrix}
\right)
q_1^2q_2
+
\left(
\begin{smallmatrix}
16(r^{-3}-9r^{-1}+16-9r+r^3)\\
-72(r^{-3}-3r^{-1}+3r-r^3)\\
+128(r^{-3}-2+r^3)\\
-144(r^{-3}+5r^{-1}-5r-r^3)\\
+128(r^{-3}-2+r^3)\\
-72(r^{-3}-3r^{-1}+3r-r^3)\\
16(r^{-3}-9r^{-1}+16-9r+r^3)\\
\end{smallmatrix}
\right)
q_1^2q_2^2
+\dots\\
\end{align*}

Using standard involutions one sees that  interchanging $q_1$ and $q_2$ 
inverts the order of the coordinates of $a(N)$, 
while interchanging $r$ and $1/r$ makes the $i$th coordinate 
change sign by $(-1)^{i}$ for $i=0,\ldots,6$.
One can read off the first non-zero Fourier coefficients:

\begin{footnotesize}
\smallskip
\vbox{
\bigskip\centerline{\def\quad{\hskip 0.6em\relax}
\def\quod{\hskip 0.5em\relax }
\vbox{\offinterlineskip
\hrule
\halign{&\vrule#&\strut\quod\hfil#\quad\cr
height2pt&\omit&&\omit&&\omit&\cr
& ${}^ta([1,0,1])$ && ${}^ta([1,1,1])$ && ${}^ta([1,0,2])$  &\cr
\noalign{\hrule}
& $(0,0,-2,0,-2,0,0)$ && $(0,0,1,2,1,0,0)$ && $(0,0,36,0,12,0,-24)$ & \cr
} \hrule}
}}
\end{footnotesize}

\begin{remark}
We can play the same game of restriction also with $\chi_{6,8}$. It restriction
to $\frak{H}_1\times \frak{H}_1$ vanishes. Its first non-vanishing term
is $t_1(\chi_{6,8})$ and it can be viewed as a section of 
$$
\oplus_{i=0}^6 \, L_1^{15-i}\otimes L_2^{9+i}
$$
with $L_1$ and $L_2$ the Hodge bundle ${\EE}_1$ on the first and second 
component. It turns out to be $(0,0,0,\Delta\otimes \Delta,0,0,0)$.
\end{remark}

The question arises what the zero locus of $\chi_{6,8}$ is. 
It contains ${\mathcal A}_{1,1}$.
Probably this is all. We prove that additionally there are only finitely
many points where $\chi_{6,8}$ vanishes.

\begin{proposition}
The zero locus of $\chi_{6,8}$ in ${\mathcal A}_2$ consists of ${\mathcal A}_{1,1}$ and
possibly finitely many isolated points.
\end{proposition}
\begin{proof}
If $V$ is the zero locus of $\chi_{6,8}$ in ${\mathcal A}_2$ 
then $V$ contains ${\mathcal A}_{1,1}$.
Let $V_1$ be the $1$-dimensional part of $V$. Then the closure  $\bar{V}_1$
of $V_1$ 
in the Satake compactification ${\mathcal A}_2^*$ intersects 
the closure $\overline{\mathcal A}_{1,1}$
of ${\mathcal A}_{1,1}$ since it is an ample divisor.
But in the neighborhood of $\overline{\mathcal A}_{1,1}$ our form is given
by $(0,0,0,\Delta\otimes \Delta,0,0,0)$ and this does not vanish 
in $U\backslash \bar{\mathcal A}_{1,1}$ with $U$ a suitable open
neighborhood of $\bar{\mathcal A}_{1,1}$ in ${\mathcal A}_2^*$.
Hence every irreducible component of 
the curve $V_1$ has to intersect $\overline{\mathcal A}_{1,1}$
at the point $(\infty,\infty)$, the $0$-dimensional cusp of 
${\mathcal A}_2^*$. But there the Fourier series shows that
$\chi_{6,8}$ does not vanish in 
$U'\backslash (U'\cap \overline{\mathcal A}_{1,1})$
for $U'$ a suitable neighborhood of the $0$-dimensional cusp of 
${\mathcal A}_2^*$.
\end{proof}

We end this section by giving a list of eigenvalues $\lambda_p$ for the Hecke
operator $T_p$ for $p$ prime $\leq 47$. These eigenvalues were obtained by
using the Fourier expansion. In an appendix we give a short
summary how to calculate such eigenvalues. 
In \cite{F-vdG} the eigenvalues $\lambda_p$ of $T(p)$ and $\lambda_{p^2}$ of $T(p^2)$ 
were calculated for $p=2,3,5$ and $7$ using a completely different 
method, namely using counting points on genus $2$ curves
 over finite fields. 
Also the eigenvalues of all $T(q)$ for $q\leq 37$ and $q\neq 27$ 
are available by \cite{F-vdG}, see also \cite{B-F-vdG}.
These values agree. Yet another way to obtain these eigenvalues is given in
\cite[Section 2, Table 3]{C} 
where Chenevier and Lannes 
use the Kneser neighborhoods of lattices to calculate these.

\begin{footnotesize}
\smallskip
\vbox{
\bigskip\centerline{\def\quad{\hskip 0.6em\relax}
\def\quod{\hskip 0.5em\relax }
\vbox{\offinterlineskip
\hrule
\halign{&\vrule#&\strut\quod\hfil#\quad\cr
height2pt&\omit&&\omit&&\omit&\cr
&$p$ && $\lambda_p$ && $\lambda_{p^2}$ & \cr
\noalign{\hrule}
& $2$ && $0$ && $-57344$ & \cr
& $3$ && $-27000$ && $143765361$ & \cr
& $5$ && $2843100$ &&$-7734928874375$ & \cr
& $7$ && $-107822000$ && $4057621173384801$ & \cr
& $11$ && $3760397784$ &&& \cr
& $13$ && $9952079500$ &&& \cr
& $17$ && $243132070500$ &&& \cr
& $19$ && $595569231400$ &&& \cr
& $23$ && $-6848349930000$ &&& \cr
& $29$ && $53451678149100$ &&& \cr
& $31$ && $234734887975744$ &&& \cr
& $37$ && $448712646713500$ &&& \cr
& $43$ && $ -1828093644641000$ &&& \cr 
& $47$ && $-6797312934516000$ &&& \cr
} \hrule}
}}
\end{footnotesize}

\end{section}
\begin{section}{The Sibling in $S_{4,0,8}(\Gamma_3)$}\label{TheFormF408}
Let $B$ (beta) denote a generator of the $1$-dimensional space
$S_{4,0,8}(\Gamma_3)$. We can find $B$ in the restriction $B \otimes \Delta$
of the Schottky form $J_8$ along $\frak{H_3}\times \frak{H}_1$.
This enables us to write the beginning of the Fourier expansion.
With the variables $q_{m}=e^{2\pi i \tau_m}$ and 
$u=e^{2\pi i \tau_{12}}$,
$v=e^{2\pi i \tau_{13}}$,
$w=e^{2\pi i \tau_{23}}$, we find the expansion
$$
B(\tau)=
\left( \begin{smallmatrix}
0\\ 0 \\ 0 \\
(v-1)^2(w-1)^2/vw \\
(u-1)(v-1)(w-1)(-1+1/vw+1/uw-1/uv) \\
(u-1)^2(w-1)^2/uw \\
0\\
(u-1)(v-1)(w-1)(-1+1/vw-1/uw+1/uv) \\
(u-1)(v-1)(w-1)(-1-1/vw+1/uw+1/uv)\\
0\\ 0 \\ 0 \\
(u-1)^2(v-1)^2/uv \\
0\\ 0\\
\end{smallmatrix}
\right )q_1q_2q_3+\cdots
$$
Using the action of $\gamma\in \Gamma_3$
that sends $\tau_{ij}$ to $\tau'=\tau_{\sigma(i),\sigma(j)}$
with $\sigma=(1\, 3) \in \frak{S}_3$ 
we see that the coordinates $v_i$ ($i=1,\ldots,15$) of the Fourier
coefficients satisfy the identities
$v_1(\tau)=v_{15}(\tau')$, $v_2(\tau)=v_{14}(\tau')$,
$v_3(\tau)=v_{10}(\tau')$,  $v_4(\tau)=v_{13}(\tau')$,
$v_5(\tau)=v_9(\tau')$, $v_6(\tau)=v_6(\tau')$,
$v_7(\tau)=v_{12}(\tau')$, $v_8(\tau)=v_8(\tau')$
One can use the Fourier coefficients
$$
\begin{aligned}
a([1\, 1\, 1\,; 0\, 0\, 0])&=^t[0, 0, 0, 4, 0, 4, 0, 0, 0, 0, 0, 0, 4, 0, 0]\\
a([2\, 2\, 2\,; 0\, 0\, 0])&=^t[-512, 0, 0, -2816, 0, -2816, 0, 0, 0, 0, -512, 0, -2816, 0, -512] \\
a([1\, 1\, 2\,; 1\, 2\, 2])&=^t[0, 0, 0, 0, 1, 1, 0, 1, 3, 2, 0, 0, 1, 2, 1] \\
a([1\, 2\, 2\,; 0\, 2\, 0])&=^t [0, 0, 0, -24, 0, -48, 0, -48, 0, -96, 48, 0, -48, 0, -48] \\
a([1\, 2\, 2\,; 2\, 0\, 0])&=^t [0, 0, 0, -48, 0, -24, -96, 0, -48, 0, -48, 0, -48, 0, 48];
\end{aligned}
$$
to calculate the Hecke eigenvalue at $2$, see the appendix. 
We have $\lambda_2=-1728$, 
and this fits because $B$ is a lift of $\Delta$ 
from $\Gamma_1$ (see \cite{B-F-vdG}) 
and its Hecke eigenvalue $\lambda_p$ is given by the formula
$$
\lambda_p= \tau(p)(p^5+\tau(p)+p^6)\, ,
$$
where $\tau(p)$ is the Hecke eigenvalue of $\Delta$ at $p$.
\end{section}
\begin{section}{Appendix: Hecke Operators for Vector valued Modular forms of 
Degree Two}
Here we give a short treatment of Hecke operators on vector-valued Siegel
modular forms of degree $2$. The basic reference is \cite{Andrianov}.

\subsection{The Hecke operator $T_p$}
For a prime $p$ we consider the double coset 
$T_p=\Gamma_2 \, {\rm{diag}}(1,1,p,p)\Gamma_2$.
Following Andrianov, we have the following left coset 
decomposition for $T_p$.

\begin{proposition}
The double coset $T_p$ admits the following left coset decomposition:
$$
\Gamma_2 
\left(
\begin{smallmatrix}
p & 0 & 0 & 0 \\
0 & p & 0 & 0 \\
0 & 0 & 1 & 0 \\
0 & 0 & 0 & 1
\end{smallmatrix}
\right)
+
\sum_{0 \leq a,b,c \leq p-1}
\Gamma_2 
\left(
\begin{smallmatrix}
1 & 0 & a & b \\
0 & 1 & b & c \\
0 & 0 & p & 0 \\
0 & 0 & 0 & p
\end{smallmatrix}
\right)
+
\sum_{0 \leq a \leq p-1}
\Gamma_2 
\left(
\begin{smallmatrix}
0 & -p & 0 & 0 \\
1 & 0 & a & 0 \\
0 & 0 & 0 & -1 \\
0 & 0 & p & 0
\end{smallmatrix}
\right)
+
\sum_{0 \leq a,m \leq p-1}
\Gamma_2 
\left(
\begin{smallmatrix}
p & 0 & 0 & 0 \\
-m & 1 & 0 & a \\
0 & 0 & 1 & m \\
0 & 0 & 0 & p
\end{smallmatrix}
\right)
$$
and we have $\deg(T_p)=p^3+p^2+p+1$.
\end{proposition}
Then we define an action of $T_p$ on $M_{j,k}(\Gamma_2)$ via 
$$
F\vert_{j,k}\, T_p=
p^{2k+j-3}
\sum_{i}
F\vert_{j,k}\, \sigma_i
$$
where for
$
 \sigma_i=
 \left(
\begin{smallmatrix}
a_i & b_i  \\
c_i & d_i \\
\end{smallmatrix}
\right)
$,
the slash operator is given by:
\[
(F\vert_{j,k}\, \sigma_i)(\tau)=
\det(c_i\tau+d_i)^{-k}
{\rm{Sym}}^j ((c_i\tau+d_i)^{-1})
F((a_i\tau+b_i)(c_i\tau+d_i)^{-1})
\]
for every $\tau\in \frak{H}_2$. 

The action on the Fourier expansion of 
$F(\tau)=\sum_{N\geq 0} a(N) e^{2 \pi i {\rm Tr} (N\tau)}$
is given by the following proposition.
If
$$
(F\vert_{j,k}\, T_p)(\tau)=
\sum_{N\geqslant 0}
b(N) e^{2 \pi i {\rm{Tr}}(N\tau)}
$$
and if we write 
$[n_1,n_{12},n_2]$ for
$
N=
 \left(
\begin{smallmatrix}
n_1 & n_{12}/2  \\
 n_{12}/2 & n_2 \\
\end{smallmatrix}
\right)
$
with $n_i\in \NN$ and $n_{12}\in \ZZ$, we have
\begin{proposition}\label{ActionTp}
The coefficient $b([n_1,n_{12},n_2])$ is given by
\begin{align*}
& p^{2k+j-3}a([\frac{n_1}{p},\frac{n_{12}}{p},\frac{n_2}{p}])+
a([pn_1,pn_{12},pn_2]) 
+p^{k+j-2} 
\Sy^j
 \left(
\begin{smallmatrix}
0 & 1/p  \\
-1 & 0 \\
\end{smallmatrix}
\right) 
a([\frac{n_2}{p},-n_{12}, pn_1])\\
&+
p^{k+j-2} 
\sum_{0 \leq m \leq p-1}
\Sy^j
 \left(
\begin{smallmatrix}
1 & -m/p  \\
0 & 1/p \\
\end{smallmatrix}
\right) 
a([\frac{n_1+n_{12}m+n_2m^2}{p},n_{12}+2n_2m, pn_2])
\end{align*}
where $a([n_1,n_{12},n_{2}])=0$ if $n_1$, $n_2$ and $n_{12}$ are not all integral.
\end{proposition}

\begin{corollary} For $N=[1,1,1]$ the Fourier coefficient $b([1,1,1])$
is given by
$$
\begin{cases}
a([3,3,3])+3^{k-2} 
\Sy^j
 \left(
\begin{smallmatrix}
3 & -1  \\
0 & 1 \\
\end{smallmatrix}
\right) 
a([1,3,3])  \, \, & {\rm{if}} \, p=3\\
a([p,p,p]) \, \, & {\rm{if}} \, p \not \equiv 1 \bmod 3\\
a([p,p,p]) 
+p^{k-2} 
\sum_{i=1}^{2}
\Sy^j
 \left(
\begin{smallmatrix}
p & -m_i  \\
0 & 1 \\
\end{smallmatrix}
\right) 
a([\frac{1+m_i+m_i^2}{p},1+2m_i, p])  \, \,&  {\rm{if}} \, p \equiv 1 \bmod 3
\end{cases}
$$
where in the latter case $m_1$ and $m_2$ are the two integers $0\leq m_i \leq p-1$ that reduce to the roots of the polynomial $1+X+X^2$ over $\FF_p$.
\end{corollary}

\begin{corollary}
For $N=[1,0,1]$, the coefficient $b([1,0,1])$ is given by
\[
\begin{cases}
a([2,0,2])+2^{k-2} 
\Sy^j
 \left(
\begin{smallmatrix}
2 & -1  \\
0 & 1 \\
\end{smallmatrix}
\right) 
a([1,2,2])  \, \,&  {\rm{if}} \, p=2\\
a([p,0,p]) \, \,&  {\rm{if}} \, p \not \equiv 1 \bmod 4\\
a([p,0,p]) 
+p^{k-2} 
\Sy^j
 \left(
\begin{smallmatrix}
p & \pm m_0  \\
0 & 1 \\
\end{smallmatrix}
\right) 
a([\frac{1+m_0^2}{p},\mp 2m_0, p])  \, \,&  {\rm{if}} \, p \equiv 1 \bmod 4
\end{cases}
\]
where in the latter case $\pm m_0$  are the two roots of the polynomial $1+X^2$ over $\FF_p$.
\end{corollary}

With these corollaries we can calculate eigenvalues $\lambda_p$ of an
eigenform $F\in M_{j,k}(\Gamma_2)$ as follows. If $p\neq 3$ and 
$p\not\equiv 1 (\bmod 3)$ and the $i$th coordinate of the vector 
$a([1,1,1])$ does not vanish then we can take 
$\lambda_p=a([p,p,p])_i/a([1,1,1])_i$,
and similarly, for $p\neq 2$ and $p\not\equiv 1 (\bmod 4)$ we can take
$\lambda_p= a([p,p,p])_i/a([1,0,1])_i$ provided that the $i$th coordinate
of $a([1,0,1])$ does not vanish. Moreover, 
if the last component of $a([1,1,1])$ is not zero and
$p=3$ or $p\equiv 1 (\bmod 3)$ we can use only this component to get
$\lambda_p$ since
$\Sy^j \left(
\begin{smallmatrix}
p & *  \\
0 & 1 \\
\end{smallmatrix}
\right)$
is upper triangular. Similarly,
if the last component of $a([1,0,1])$ is not zero and 
$p=2$ or $p \equiv 1 \bmod 4$, we can use only this component in order to
compute $\lambda_p$ for the same reason as before.

\subsection{The Hecke operator $T_{p^2}$}
The Hecke operator $T_{p^2}$ 
is defined via the double coset
$$
T_{p^2}=
\Gamma_2
 \left(
\begin{smallmatrix}
p & 0 & 0 & 0 \\
0 & p & 0 & 0 \\
0 & 0 & p & 0 \\
0 & 0 & 0 & p
\end{smallmatrix}
\right)
 \Gamma_2
 +
\Gamma_2
 \left(
\begin{smallmatrix}
1 & 0 & 0 & 0 \\
0 & p & 0 & 0 \\
0 & 0 & p^2 & 0 \\
0 & 0 & 0 & p
\end{smallmatrix}
\right)
 \Gamma_2
 +
\Gamma_2
 \left(
\begin{smallmatrix}
1 & 0 & 0 & 0 \\
0 & 1 & 0 & 0 \\
0 & 0 & p^2 & 0 \\
0 & 0 & 0 & p^2
\end{smallmatrix}
\right)
 \Gamma_2\\
$$

\begin{proposition}
The Hecke operator $T_{p^2}$ 
has degree $p^6+p^5+2p^4+2p^3+p^2+p+1$ and 
admits the following left coset decomposition
\begin{tiny}
\begin{align*}
&
\Gamma_2 
\left(
\begin{smallmatrix}
p^2 & 0 & 0 & 0 \\
0 & p^2 & 0 & 0 \\
0 & 0 & 1 & 0 \\
0 & 0 & 0 & 1
\end{smallmatrix}
\right) +
\sum_{0 \leq a,b,c \leq p-1}
\Gamma_2 
\left(
\begin{smallmatrix}
p & 0 & a & b \\
0 & p & b & c \\
0 & 0 & p & 0 \\
0 & 0 & 0 & p
\end{smallmatrix}
\right) +
\sum_{0 \leq a,b,c \leq p^2-1}
\Gamma_2 
\left(
\begin{smallmatrix}
1 & 0 & a & b \\
0 & 1 & b & c \\
0 & 0 & p^2 & 0 \\
0 & 0 & 0 & p^2
\end{smallmatrix}
\right) \\&
+
\sum_{0 \leq a \leq p-1 }
\Gamma_2 
\left(
\begin{smallmatrix}
p & 0 & a & 0 \\
0 & p^2 & 0 & 0 \\
0 & 0 & p & 0 \\
0 & 0 & 0 & 1
\end{smallmatrix}
\right)
+
\sum_{0 \leq a, m \leq p-1 }
\Gamma_2 
\left(
\begin{smallmatrix}
p^2 & 0 & 0 & 0 \\
-mp & p & 0 & a \\
0 & 0 & 1 & m \\
0 & 0 & 0 & p
\end{smallmatrix}
\right)
+
\sum_{\substack{ 0 \leq a, b \leq p-1 \\ 0 \leq c \leq p^2-1}}
\Gamma_2 
\left(
\begin{smallmatrix}
1 & 0 & c & -b \\
0 & p & -pb & a \\
0 & 0 & p^2 & 0 \\
0 & 0 & 0 & p
\end{smallmatrix}
\right) \\
&+
\sum_{\substack{ 0 \leq a, b, m \leq p-1 \\ 0 \leq c \leq p^2-1}}
\Gamma_2 
\left(
\begin{smallmatrix}
p & 0 & a & am+bp \\
-m & 1 & b & bm+c \\
0 & 0 & p & pm \\
0 & 0 & 0 & p^2
\end{smallmatrix}
\right)
+
\sum_{\substack{0 \leq a, m \leq p^2-1 }}
\Gamma_2 
\left(
\begin{smallmatrix}
p^2 & 0 & 0 & 0 \\
-m & 1 & 0 & a \\
0 & 0 & 1 & m \\
0 & 0 & 0 & p^2
\end{smallmatrix}
\right)
+
\sum_{\substack{0 \leq a \leq p^2-1 \\ 0 \leq n \leq p-1 }}
\Gamma_2 
\left(
\begin{smallmatrix}
1 & np & a & 0 \\
0 & p^2 & 0 & 0 \\
0 & 0 & p^2 & 0 \\
0 & 0 & -np & 1
\end{smallmatrix}
\right).
\end{align*}
\end{tiny}
\end{proposition}

We now consider the action on a modular form $F \in M_{j,k}(\Gamma_2)$
defined by
$$
F|_{j,k}T_{p^2}= p^{4k+2j-6} \sum_i F|_{j,k} \sigma_i
$$
where the sum is over the left coset representatives. 
If $F=\sum_{N\geq 0} a(N) e^{2\pi i {\rm Tr}(N\tau)}$ 
the result is a modular
form with Fourier expansion 
$\sum_{N\geq 0} c(N) e^{2 \pi i {\rm Tr} (N \tau)}$
with the $c(N)=a_{p^2}(N)$ expressed in the $a(N)$ as follows.

\begin{proposition}\label{ActionTpsquare}
Let $F\in M_{j,k}(\Gamma_2)$ with Fourier expansion
\[
F(\tau)=
\sum_{N\geqslant 0}
a(N) e^{2 \pi i {\rm{Tr}}(N\tau)}
\quad
{\rm{and \, write}}
\quad
F\vert_{j,k}\, T_{p^2}(\tau)=
\sum_{N\geqslant 0}
a_{p^2}(N) e^{2 \pi i {\rm{Tr}}(N\tau)}.
\]
Then we have ($N=[n_1,n_{12},n_2]$)
\begin{align*}
a_{p^2}(N)=&
a(p^2N)+p^{4k+2j-6}a(\frac{N}{p^2})+p^{2k+j-3}a(N)\delta_p(n_1, n_{12}, n_2)\\
&+
p^{3k+j-5}
\Sy^j\left(\begin{smallmatrix}1 & 0 \\0 & p\end{smallmatrix}\right)
a([n_1,n_{12}/p,n_2/p^2])\delta_p(n_1)\\
&+
p^{3k+j-5}
\sum_{0 \leq m \leq p-1}
\Sy^j\left(\begin{smallmatrix} p & -m \\ 0 & 1 \end{smallmatrix}\right)
a([\frac{n_1+n_{12}m+n_2m^2}{p^2},\frac{n_{12}+2n_2m}{p},n_2])\delta_p(n_2)\\
&+
p^{k-2}
\Sy^j\left(\begin{smallmatrix}1 & 0 \\0 & p\end{smallmatrix}\right)
a([p^2n_1,pn_{12},n_2])\delta_p(n_2)\\
&+
p^{k-2}
\sum_{\substack{0 \leq m \leq p-1\\ n_1+n_{12}m+n_2m^2 \equiv 0 \bmod p}}
\hspace{-15pt}
\Sy^j\left(\begin{smallmatrix} p & -m \\0 & 1\end{smallmatrix}\right)
a([n_1+n_{12}m+n_2m^2 ,p(n_{12}+2n_2m),p^2n_2])\\
&+
p^{2k-4}
\sum_{0 \leq m \leq p^2-1}
\Sy^j\left(\begin{smallmatrix} p^2 & -m \\0 & 1\end{smallmatrix}\right)
a([\frac{n_1+n_{12}m+n_2m^2}{p^2},n_{12}+2n_2m,p^2n_2])\\
&+
p^{2k-4}
\sum_{0 \leq n \leq p-1}
\Sy^j\left(\begin{smallmatrix} 1 & 0 \\ np & p^2 \end{smallmatrix}\right)
a([p^2n_1,n_{12}-2pn_1n,\frac{n_1p^2n^2-n_{12}np+n_2}{p^2} ])
\end{align*}
where $a([n_1,n_{12},n_2])=0$ if $n_1, n_{12}$ and $n_2$ are not all integral.
\end{proposition}

From this proposition, we deduce the following corollary:
\begin{corollary}\label{coro}
Let $F\in M_{j,k}(\Gamma_2)$ be a Hecke eigenform with eigenvalue 
$\lambda_{p^2}$ at $p^2$.
Assume that $a([1,1,1])\neq 0$.
Then $\lambda_{p^2}(F)a([1,1,1])$ equals
\begin{align*}
 a(p^2N)+ &
p^{k-2}
\sum_{\substack{0 \leq m \leq p-1\\ 1+m+m^2 \equiv 0 \bmod p}}
\hspace{-15pt}
\Sy^j\left(\begin{smallmatrix} p & -m \\0 & 1\end{smallmatrix}\right)
a([1+m+m^2 ,p(1+2m),p^2])\\
&+
p^{2k-4}
\sum_{0 \leq m \leq p^2-1}
\Sy^j\left(\begin{smallmatrix} p^2 & -m \\0 & 1\end{smallmatrix}\right)
a([\frac{1+m+m^2}{p^2},1+2m,p^2]).
\end{align*}
\end{corollary}

\end{section}
\begin{section}{Appendix: Hecke Operators for Vector valued Modular forms of 
Degree Three}

Here we treat only the case of the Hecke operator $T_p$ for $p$ a prime.
Here $T_p$ is the double coset $\Gamma_3 \, {\rm diag}(1,1,1,p,p,p)\Gamma_3$.

\begin{proposition}
We have $\deg(T_p)=p^6+p^5+p^4+2\, p^3+p^2+p+1$ and there is the following
left coset decomposition.
\begin{align*}
&
\Gamma_3
\left(
\begin{smallmatrix}
p & 0 & 0 & 0 & 0 & 0 \\
0 & p & 0 & 0 & 0 & 0 \\
0 & 0 & p & 0 & 0 & 0 \\
0 & 0 & 0 & 1 & 0 & 0 \\
0 & 0 & 0 & 0 & 1 & 0 \\
0 & 0 & 0 & 0 & 0 & 1 \\
\end{smallmatrix}
\right)
+
\sum_{0 \leq a, \dots, f \leq p-1}
\Gamma_3 
\left(
\begin{smallmatrix}
1 & 0 & 0 & a & b & c \\
0 & 1 & 0 & b & d & e \\
0 & 0 & 1 & c & e & f \\
0 & 0 & 0 & p & 0 & 0 \\
0 & 0 & 0 & 0 & p & 0 \\
0 & 0 & 0 & 0 & 0 & p \\
\end{smallmatrix}
\right)
+
\sum_{0 \leq a,u,v  \leq p-1}
\Gamma_3
\left(
\begin{smallmatrix}
p & 0 & 0 & 0 & 0 & 0 \\
0 & p & 0 & 0 & 0 & 0 \\
-u & -v & 1 & 0 & 0 & a \\
0 & 0 & 0 & 1 & 0 & u \\
0 & 0 & 0 & 0 & 1 & v \\
0 & 0 & 0 & 0 & 0 & p \\
\end{smallmatrix}
\right)\\
&+
\sum_{0 \leq a,u \leq p-1}
\Gamma_3
\left(
\begin{smallmatrix}
p & 0 & 0 & 0 & 0 & 0 \\
0 & 0 & p & 0 & 0 & 0 \\
-u & 1 & 0 & 0 & a & 0 \\
0 & 0 & 0 & 1 & u & 0 \\
0 & 0 & 0 & 0 & 0 & 1 \\
0 & 0 & 0 & 0 & p & 0 \\
\end{smallmatrix}
\right)
+
\sum_{0 \leq a \leq p-1}
\Gamma_3
\left(
\begin{smallmatrix}
0 & 0 & p & 0 & 0 & 0 \\
0 & p & 0 & 0 & 0 & 0 \\
1 & 0 & 0 & a & 0 & 0 \\
0 & 0 & 0 & 0 & 0 & 1 \\
0 & 0 & 0 & 0 & 1 & 0 \\
0 & 0 & 0 & p & 0 & 0 \\
\end{smallmatrix}
\right)
+
\sum_{0 \leq a,b,c,u,v  \leq p-1}
\Gamma_3
\left(
\begin{smallmatrix}
p & 0 & 0 & 0 & 0 & 0 \\
-u & 1 & 0 & 0 & a & b \\
-v & 0 & 1 & 0 & b & c \\
0 & 0 & 0 & 1 & u & v \\
0 & 0 & 0 & 0 & p & 0 \\
0 & 0 & 0 & 0 & 0 & p \\
\end{smallmatrix}
\right) \\
&+
\sum_{0 \leq a,b,c,u  \leq p-1}
\Gamma_3
\left(
\begin{smallmatrix}
0 & p & 0 & 0 & 0 & 0 \\
1 & 0 & 0 & a & 0 & b \\
0 & -u & 1 & b & 0 & c \\
0 & 0 & 0 & 0 & 1 & u \\
0 & 0 & 0 & p & 0 & 0 \\
0 & 0 & 0 & 0 & 0 & p \\
\end{smallmatrix}
\right)
+
\sum_{0 \leq a,b,c  \leq p-1}
\Gamma_3
\left(
\begin{smallmatrix}
0 & 0 & p & 0 & 0 & 0 \\
0 & 1 & 0 & a & b & 0 \\
1 & 0 & 0 & c & a & 0 \\
0 & 0 & 0 & 0 & 0 & 1 \\
0 & 0 & 0 & 0 & p & 0 \\
0 & 0 & 0 & p & 0 & 0
\end{smallmatrix}
\right)
\end{align*}
\end{proposition}

We are now able to describe the action of the operator $T_p$ on the Fourier expansion of a modular form
on $\Gamma_3$. Let $F\in M_{i,j,k}(\Gamma_3)$ with Fourier expansion
$
F(\tau)=\sum_{N\geq 0} a(N)e^{2 \pi i {\rm{Tr}}(N\tau)}.
$
The action of the operator $T_p$ on $F$ reads as follows:
\[
F\vert_{i,j,k}\, T_p=
p^{i+2j+3k-6}
\sum_{i}
F\vert_{i,j,k}\, \sigma_i
\]
where for
$
 \sigma_i=
 \left(
\begin{smallmatrix}
a_i & b_i  \\
c_i & d_i \\
\end{smallmatrix}
\right)
$,
the slash operator is given by:
\[
(F\vert_{i,j,k}\, \sigma_i)(\tau)=
\det(c_i\tau+d_i)^{-k}
\Sy^i ((c_i\tau+d_i)^{-1})
\Sy^j (\wedge^2((c_i\tau+d_i)^{-1})))
F((a_i\tau+b_i)(c_i\tau+d_i)^{-1})
\]
for every $\tau\in \frak{H}_3$.  Let us write
\[
(F\vert_{i,j,k}\, T_p)(\tau)=
\sum_{N\geqslant 0}
a_{p}(N) e^{2 \pi i {\rm{Tr}}(N\tau)}.
\]
We will use two kinds of notations for $
N=
\left(
\begin{smallmatrix}
n_1 & n_{12}/2 & n_{13}/2  \\
n_{12}/2  & n_2 & n_{23}/2   \\
n_{13}/2  & n_{23}/2  & n_3 
\end{smallmatrix}
\right)
$,
namely
\[
N=
\left(
\begin{smallmatrix}
n_1 & n_{12}/2 & n_{13}/2  \\
n_{12}/2  & n_2 & n_{23}/2   \\
n_{13}/2  & n_{23}/2  & n_3 
\end{smallmatrix}
\right)
\leftrightarrow
[n_1\, n_2\, n_3\, ; n_{12}\, n_{13}\, n_{23}]
\leftrightarrow
\left[
\begin{smallmatrix}
n_1 \\
n_2 \\
n_3 \\
n_{12} \\
n_{13} \\
n_{23} \\
\end{smallmatrix}
\right].
\]
We also use the notations
\[
\Sy^i(a^{-1})=\Sy^{-i}(a) \quad
{\rm{and}} \quad
\Sy^j(\wedge^2(a^{-1}))=\Sy^{-j}(\wedge^2(a)).
\]

\begin{proposition}
We have
\begin{tiny}
\begin{align*}
a_p(N)=& \, 
p^{i+2j+3k-6}a(N/p)
+
a(p\, N)\\
&+
p^{i+2j+2k-5}
\hspace{-8pt}
\sum_{0\leq u,v\leq p-1}
\hspace{-10pt}
\Sy^{-i}
\left(
\begin{smallmatrix}
1 & 0 & u  \\
0 & 1 & v  \\
0 & 0 & p 
\end{smallmatrix}
\right)
\Sy^{-j}(\wedge^2(
\left(
\begin{smallmatrix}
1 & 0 & u  \\
0 & 1 & v  \\
0 & 0 & p 
\end{smallmatrix}
\right))
a(\left[
\begin{smallmatrix}
(n_1+n_{13}u+n_3u^2)/p \\
(n_2+n_{23}v+n_3v^2)/p  \\
pn_3 \\
(n_{12}+n_{23}u+n_{13}v+2n_3uv)/p \\
n_{13}+2n_3u \\
n_{23}+2n_3v 
\end{smallmatrix}
\right])\\
&+
p^{i+2j+2k-5}
\hspace{-8pt}
\sum_{0\leq u \leq p-1}
\hspace{-10pt}
\Sy^{-i}
\left(
\begin{smallmatrix}
1 & u & 0  \\
0 & 0 & 1  \\
0 & p & 0 
\end{smallmatrix}
\right)
\Sy^{-j}(\wedge^2(
\left(
\begin{smallmatrix}
1 & u & 0  \\
0 & 0 & 1  \\
0 & p & 0 
\end{smallmatrix}
\right))
a(\left[
\begin{smallmatrix}
(n_1+n_{12}u+n_2u^2)/p \\
n_3/p  \\
pn_2 \\
(n_{13}+2n_{23}u)/p \\
n_{12}+2n_2u \\
n_{23}
\end{smallmatrix}
\right])\\
&+
p^{i+2j+2k-5}
\Sy^{-i}
\left(
\begin{smallmatrix}
0 & 0 & 1  \\
0 & 1 & 0  \\
p & 0 & 0 
\end{smallmatrix}
\right)
\Sy^{-j}(\wedge^2(
\left(
\begin{smallmatrix}
0 & 0 & 1  \\
0 & 1 & 0  \\
p & 0 & 0 
\end{smallmatrix}
\right))
a([n_3/p\,n_2/p\, pn_1\,; n_{23}/p\, n_{13}/p\, n_{23}])\\
&+
p^{i+2j+k-3}
\hspace{-8pt}
\sum_{0\leq u,v\leq p-1}
\hspace{-10pt}
\Sy^{-i}
\left(
\begin{smallmatrix}
1 & u & v  \\
0 & p & 0  \\
0 & 0 & p 
\end{smallmatrix}
\right)
\Sy^{-j}(\wedge^2(
\left(
\begin{smallmatrix}
1 & u & v  \\
0 & p & 0  \\
0 & 0 & p 
\end{smallmatrix}
\right))
a(\left[
\begin{smallmatrix}
(n_1+n_{12}u+n_{13}v+n_2u^2+n_{23}uv+n_3v^2)/p \\
p\,n_2  \\
p\,n_3 \\
n_{12}+2n_2u+n_{23}v \\
n_{13}+n_{23}u +2n_3v\\
p\, n_{23} 
\end{smallmatrix}
\right])\\
&+
p^{i+2j+k-3}
\hspace{-8pt}
\sum_{0\leq u\leq p-1}
\hspace{-10pt}
\Sy^{-i}
\left(
\begin{smallmatrix}
0 & 1 & u  \\
p & 0 & 0  \\
0 & 0 & p 
\end{smallmatrix}
\right)
\Sy^{-j}(\wedge^2(
\left(
\begin{smallmatrix}
0 & 1 & u  \\
p & 0 & 0  \\
0 & 0 & p 
\end{smallmatrix}
\right))
a(\left[
\begin{smallmatrix}
(n_2+n_{23}u+n_3u^2)/p \\
p\,n_1  \\
p\,n_3 \\
n_{12}+n_{13}u \\
n_{23}+2n_3u\\
p\, n_{13} 
\end{smallmatrix}
\right])\\
&+
p^{i+2j+k-3}
\Sy^{-i}
\left(
\begin{smallmatrix}
0 & 0 & 1  \\
0 & p & 0  \\
p & 0 & 0 
\end{smallmatrix}
\right)
\Sy^{-j}(\wedge^2(
\left(
\begin{smallmatrix}
0 & 0 & 1  \\
0 & p & 0  \\
p & 0 & 0 
\end{smallmatrix}
\right))
a([n_3/p\,pn_2\, pn_1\,; n_{23}\, n_{13}\, pn_{12}])\\
\end{align*}
\end{tiny}
where $a([n_1\, n_2\, n_3\, ; n_{12}\, n_{13}\, n_{23}])=0$ if $n_1,\ldots, n_{23}$ are not all integral.
\end{proposition}

\end{section}

\end{document}